\documentclass[11pt]{article}
\usepackage{amsmath, latexsym, amsfonts, amssymb, amsthm, amscd}
\usepackage[letterpaper, margin=1in]{geometry}
\usepackage{upquote}
\usepackage{color}
\usepackage{setspace}
\usepackage[labelfont=bf]{caption}
\usepackage{stmaryrd}

\usepackage[utf8]{inputenc}
\usepackage[T1]{fontenc}
\usepackage[english]{babel}
\usepackage{dsfont}
\usepackage{mathtools} 
\usepackage{bm}

\usepackage{xcolor}
\definecolor{Maroon}{HTML}{ad2231}
\definecolor{webgreen}{HTML}{008000}

\usepackage{hyperref}
\hypersetup{colorlinks, breaklinks, urlcolor=Maroon, linkcolor=Maroon, citecolor=webgreen} 

\allowdisplaybreaks

\usepackage{bookmark}



\newtheorem{lemma}{Lemma}
\newtheorem{remark}{Remark}

\newtheorem{theorem}{Theorem}

\theoremstyle{definition}

\newtheorem*{general*}{A general remark}

\usepackage{bigints}

\usepackage[authormarkup=none]{changes}

\makeatletter
\DeclareRobustCommand{\cev}[1]{%
  {\mathpalette\do@cev{#1}}%
}
\newcommand{\do@cev}[2]{%
  \vbox{\offinterlineskip
    \sbox\z@{$\m@th#1 x$}%
    \ialign{##\cr
      \hidewidth\reflectbox{$\m@th#1\vec{}\mkern4mu$}\hidewidth\cr
      \noalign{\kern-\ht\z@}
      $\m@th#1#2$\cr
    }%
  }%
}
\makeatother

\usepackage{cleveref}

\usepackage{mathtools}
\usepackage{stmaryrd}

\newcommand\dsE{{\mathbb E}}

\newcommand\dsN{{\mathbb N}}

\newcommand\dsP{{\mathbb P}}

\newcommand\dsR{{\mathbb R}}

\newcommand\dsT{{\mathbb T}}

\newcommand\dsZ{{\mathbb Z}}

\newcommand\scE{{\mathcal E}}
\newcommand\scF{{\mathcal F}}

\newcommand\scK{{\mathcal K}}

\newcommand\scV{{\mathcal V}}

\newcommand\vect{{\bf t}}
\newcommand\vecu{{\bf u}}
\newcommand\vecv{{\bf v}}
\newcommand\vecw{{\bf w}}
\newcommand\vecx{{\bf x}}
\newcommand\vecy{{\bf y}}

\newcommand\vecU{{\bf U}}

\newcommand{\E}[1]{{\mathbb E}\left[#1\right]}
\newcommand{\e}{{\mathbb E}}

\newcommand{\p}[1]{{\mathbb P}\left(#1\right)}

\newcommand\inlawHIGH{\,{\buildrel d \over \rightarrow}\,} 
\newcommand\inlaw{{\inlawHIGH}}

\newcommand\given[1][]{\:#1\vert\:}

\usepackage[authormarkup=none]{changes}

\definechangesauthor[name={Cai},color=red]{C}

\begin{document}
\title{The \(k\)-cut model in deterministic and random trees}
\author{Gabriel Berzunza\footnote{E-mail:
        \href{mailto:gabriel.berzunza-ojeda@math.uu.se}{gabriel.berzunza-ojeda@math.uu.se}},\,\,
    Xing Shi Cai\footnote{E-mail: \href{mailto:xingshi.cai@math.uu.se}{xingshi.cai@math.uu.se}} \,  and
    Cecilia Holmgren\footnote{E-mail: \href{mailto:cecilia.holmgren@math.uu.se}{cecilia.holmgren@math.uu.se}} \\ \vspace*{10mm}
{\small Department of Mathematics, Uppsala University, Sweden} }
\maketitle

\vspace{0.1in}

\begin{abstract} 
    The \(k\)-cut number of rooted graphs was introduced by Cai et al.\ \cite{Cai2018} as a
    generalization of the classical cutting model by Meir and Moon \cite{Meir1970}. In this paper,
    we show that all moments of the \(k\)-cut number of conditioned Galton-Watson tree converge
    after proper rescaling, which implies convergence in distribution to the same limit law
    regardless of the offspring distribution of the trees. This
    extends the result of Janson \cite{Janson2006}.  Using the same method, we also show that the
    \(k\)-cut number of various random or deterministic trees of logarithmic height converges in
    probability to a constant after rescaling, such as random split-trees, uniform random recursive trees, and scale-free random trees. 
\end{abstract}



\section{Introduction and main result}

In order to measure the difficulty for the destruction of a resilient network Cai et al.\
\cite{Cai2018} introduced a generalization of the cut model of Meir and Moon  \cite{Meir1970} where
each vertex (or edge) needs to be cut $k \in \dsN$ times (instead of only once) before it is
destroyed. More precisely, consider that the resilient network is a rooted tree $\dsT_{n}$, with $n
\in \dsN$ vertices. We assume that sibling vertices in \(\dsT_{n}\) are ordered.  (Such trees
sometimes are referred to as plane trees.) We destroy it by removing its vertices as follows:  {\bf
Step 1:} Choose a vertex uniformly at random from the component that contains the root and cut the
selected vertex once. {\bf Step 2:} If this vertex has been cut $k$ times, remove the vertex
together with the edges attached to it from the tree. {\bf Step 3:} If the root has been removed,
then stop. Otherwise, go to step {\bf Step 1}. We let $\mathcal{K}_{k}(\dsT_{n})$ denote the
(random) total number of cuts needed to end this procedure the $k$-cut number, i.e.,
$\mathcal{K}_{k}(\dsT_{n})$ models how much effort it takes to destroy the network. (For simplicity,
we will omit the subscript \(k\) and write \(\mathcal{K}(\dsT_{n})\).) It should be clear that one
can define analogously an edge deletion version of the previous algorithm, where one needs to cut an
edge $k$ times before removing it from the root component. Then, one would be interested in the
number $\mathcal{K}_{e}(\dsT_{n})$ of edge cuts needed to isolate the root of $\dsT_{n}$. 

The case $k = 1$ (i.e., the traditional cutting model of Meir and Moon \cite{Meir1970}) has been
well-studied by several authors. More precisely, Meir and Moon estimated the
first and second moment of the $1$-cut number in the cases when $\dsT_{n}$ is a Cayley tree
\cite{Meir1970} and a recursive tree \cite{Meir1974}. Subsequently, several weak limit theorems for
the $1$-cut number have been obtained for Cayley trees (Panholzer \cite{Pan2004, Pan2006}), complete
binary trees (Janson \cite{Janson2004}), conditioned Galton-Watson trees (Janson \cite{Janson2006}
and Addario-Berry et al.\ \cite{Adda2014}), recursive trees (Drmota et al.\ \cite{Drmota2009},
Iksanov and Möhle \cite{Iksanov2007}), binary search trees (Holmgren \cite{Holmgren2010}) and split
trees (Holmgren \cite{Holmgren2011}). In the general case $k \geq 1$, the authors in \cite{Cai2018}
established first moment estimates of $\mathcal{K}(\dsT_{n})$ for families of
deterministic and random trees, such as paths, complete binary trees, split trees, random
recursive trees and conditioned Galton-Watson trees. In particular, the authors in \cite{Cai2018} have proven a
weak limit theorem for $\mathcal{K}(\dsT_{n})$ when $\dsT_{n}$ is a path
consisting of $n$ vertices. More recently,  Cai and Holmgren \cite{Cai20182} also obtained a weak
limit theorem in the case when $\dsT_{n}$ is a complete binary tree. 

In this work, we continue the investigation of this general cutting-down procedure in conditioned Galton-Watson trees and show that   $\scK(\dsT_{n})$, after a proper rescaling,
converges in distribution to a non-degenerate random variable. More precisely, let $\xi$ be a
non-negative integer-valued random variable such that 
\begin{eqnarray} \label{eq11}
 \dsE[\xi] = 1 \hspace*{4mm} \text{and} \hspace*{4mm} 0 < \sigma^{2} \coloneqq Var(\xi) < \infty.
\end{eqnarray} 

\noindent We further assume that the distribution of $\xi$ is aperiodic. This last condition is to
avoid unnecessary complications, but our results can be extended to the periodic case. We then
consider a Galton-Watson process with (critical) offspring distribution $\xi$. Let $\dsT_{n}$ be the
family tree conditioned on its number of vertices being $n \in \dsN$, providing that this
conditioning makes sense. The main result of this paper is the following. We write $\inlaw$ to
denote convergence in distribution.  (In the rest of the paper CRT stands for Continuum Random
Tree.)


\begin{theorem} \label{theo1}
    Let \(k \in \dsN\).
Let $\dsT_{n}$ be a Galton-Watson tree conditioned on its number of vertices being $n \in \dsN$ with offspring distribution $\xi$ satisfying \eqref{eq11}. Then, 
\begin{eqnarray} \label{eq21}
   \sigma^{-1/k} n^{-1+1/2k} \scK(\dsT_{n})   
    \inlaw Z_{\rm CRT}, \hspace*{4mm} \text{as} \hspace*{2mm} n \rightarrow \infty,  
\end{eqnarray}  

\noindent where $Z_{\rm CRT}$ is a non-degenerate random variable whose law is determined entirely by its
moments: $\dsE[Z_{\rm CRT}^{0}]=1$, and for $q \in \dsN$, $\dsE[Z_{\rm CRT}^{q}] = \eta_{k,
q}$ with
\begin{eqnarray}\label{eq:ZCRT}
\eta_{k, q}
\coloneqq
q! \int_{0}^{\infty} \cdots \int_{0}^{\infty} y_{1} (y_{1}+y_{2})\cdots (y_{1}+ \cdots +y_{q}) e^{- \frac{(y_{1}+ \cdots +y_{q})^{2}}{2} } F_{q}(\vecy_{q})\;{\rm d} y_{q} \cdots\;{\rm d} y_{1},
\end{eqnarray}

\noindent where $\vecy_{q} = (y_{1}, \dots, y_{q}) \in \dsR_{+}^{q}$ and
\begin{eqnarray*}
F_{q}(\vecy_{q}) \coloneqq  \int_{0}^{\infty} \int_{0}^{x_{1}} \dots \int_{0}^{x_{q-1}} \exp \left( - \frac{y_{1}x_{1}^{k} + y_{2}x_{2}^{k} +\dots + y_{q}x_{q}^{k} }{k !} \right)\;{\rm d} x_{q} \cdots\;{\rm d} x_{2}\;{\rm d} x_{1}.
\end{eqnarray*}

\noindent Furthermore, if $\dsE[\xi^{p}] < \infty$ for every $p \in \dsZ_{\ge 0}$, then for every $q \in \dsZ_{\ge 0}$, $\sigma^{-q/k}n^{-q+q/2k}\dsE[\scK(\dsT_{n})^{q}] \rightarrow \dsE[Z_{{\rm CRT}}^{q}]$ as $n \rightarrow \infty$. 
\end{theorem}

In the case $k=1$, \Cref{theo1} reduces to a \(Z_{{\rm CRT}}\) having a Rayleigh distribution with
density $xe^{-x^{2}/2}$, for $x \in \mathbb{R}_{+}$. More precisely, one can verify that
\(\eta_{1,q} = 2^{q/2} \Gamma(1 +q/2) \), for $q \in \dsZ_{\ge 0}$, which are the moments of a
random variable with the Rayleigh distribution; in this paper $\Gamma(\cdot)$ denotes the well-known
gamma function. As we mentioned earlier, the case $k=1$ has been shown in \cite[Theorem 1.6]{Janson2006} (or Addario-Berry et al.\ \cite{Adda2014}). We henceforth assume throughout this
paper that \(k \ge 2\). 

It is also important to mention that we could not find a simpler expression
(in general) for the moments $\eta_{k,q}$ except for some particular instances. For $q=1$, we have 
\begin{equation*}\label{eq:eta:1}
        \eta _{k,1}= 2^{-\frac{1}{2k}}\frac{(k!)^{\frac{1}{k}}}{k} \Gamma
            \left(\frac{1}{k}\right) \Gamma \left(1-\frac{1}{2 k}\right).
 \end{equation*}

\noindent Then \Cref{theo1} provides a proof of \cite[Lemma 4.10]{Cai2018}, where an estimation of
the first moment of $\scK(\dsT_{n})$ was first announced but whose proof was left to the reader.
One can also compute with the help of Mathematica the second moment of \(Z_{{\rm
CRT}}\) or other particular examples. However, the expressions are too involved and we decided
not to include them.  

On the other hand, let \( (U_{1}, \dots, U_{q}) \) be $q$ i.i.d.\ leaves of a
Brownian CRT and define the vector $(L_{0}^{\rm CRT}, L_{1}^{\rm CRT}, \dots, L_{q}^{\rm CRT})$
where \(L_{0}^{\rm CRT} = 0\) and \(L_{i}^{\rm CRT}\) is the total length of the minimal subtree of a Brownian CRT which connects its root and the leaves of \(U_{1},\dots,U_{i}\); see \cite[Lemma 21]{Aldous1993} from where one can deduce explicitly the distribution of $(L_{0}^{\rm CRT}, L_{1}^{\rm CRT}, \dots L_{q}^{\rm CRT})$. From the
proof of \Cref{theo1}, we obtain, for $q \in \dsN$, that 
\begin{eqnarray*} 
\eta_{k,q}  = q! \int_{0}^{\infty} \int_{0}^{x_{1}} \dots \int_{0}^{x_{q-1}}
\mathbb{E}\left[ 
\exp \left( 
    - \frac{\sum_{i=1}^{q} (L_{i}^{\rm CRT} - L_{i-1}^{\rm CRT})x_{i}^{k}}{k !} 
\right)
\right] \;{\rm d} \cev{\vecx}_{q} , 
\end{eqnarray*}

\noindent  where $\cev{\vecx}_{q} = (x_{q}, \dots, x_{1}) \in \dsR_{+}^{q}$. This suggests that it ought to be possible to build the random
variable $Z_{\rm CRT}$ by some construction that can be interpreted as the $k$-cut model on the
Brownian CRT defined by Aldous \cite{Aldous19912, Aldous1993}. The appearance of the Brownian CRT in
this framework should not come as a surprise since it is well-known that if we assign length
$n^{-1/2}$ to each edge of the Galton-Watson tree $\mathbb{T}_{n}$, then the latter converges weakly
to a Brownian CRT as $n \rightarrow \infty$. We believe that this connection can be exploited even
more than the one used in this work in order to obtain the precise distribution of $Z_{\rm CRT}$.
For example, ideas from \cite{Berotoin20013} and \cite{Adda2014} could be useful to answer this
question. 
 
The approach used in this work consists of implementing an extension of the idea of Janson
\cite{Janson2006}, which was used in \cite{Cai2018}, in order to study the $k$-cut model on
deterministic and random trees. The authors in \cite{Cai2018} introduced an equivalent model that
allows them to define $\scK(\dsT_{n})$ in terms of the number of records in $\dsT_{n}$ when vertices
are assigned random labels. More precisely, let $(E_{i,v})_{i\geq 1,v \in  \dsT_{n}}$ be a sequence
of independent exponential random variables with parameter $1$; ${\rm Exp}(1)$ for
short. Let $G_{r,v} \coloneqq \sum_{1 \leq i \leq r} E_{i,v}$, for  $r \in \dsN$ and $v \in
\dsT_{n}$. Clearly, $G_{r,v}$ has a gamma distribution with parameters $(r,1)$, which we
denote by Gamma$(r)$. Imagine that each vertex  $v \in \dsT_{n}$ has an alarm clock and $v$'s
clock fires at times $(G_{r,v})_{r\geq 1}$. If we cut a vertex when its alarm clock fires, then due
to the memoryless property of exponential random variables, we are actually choosing a vertex
uniformly at random to cut. However, this also means that we are cutting vertices that have already
been removed from the tree. Thus, for a cut on vertex $v$ at time $G_{r,v}$ (for some $r \in \{1,
\dots, k\}$) to be counted in $\scK(\dsT_{n})$, none of its strict ancestors can already have
been cut $k$ times, i.e.,
\begin{eqnarray*}
G_{r,v} < \min \{ G_{k,u}: u \in \dsT_{n}  \, \, \text{and} \, \,  u \,\, \text{is a strict ancestor of} \,\, v \}.
\end{eqnarray*}

\noindent When the previous event happens, we say that \(G_{r,v}\), or simply $v$, is an $r$-record and let  
\begin{eqnarray} \label{eq2}
I_{r,v} \coloneqq  \llbracket G_{r,v} < \min \{ G_{k,u}: u \in \dsT_{n}  \, \, \text{and} \, \,  u \,\, \text{is a strict ancestor of} \,\, v \} \rrbracket,
\end{eqnarray}

\noindent where $\llbracket \cdot \rrbracket$ denotes the Iverson bracket, i.e., $\llbracket S \rrbracket = 1$ if the statement $S$ is true and $\llbracket S \rrbracket = 0$ otherwise. Let $\scK_{r}(\dsT_{n})$ be the number of $r$-records, i.e., $\scK_{r}(\dsT_{n}) 
\coloneqq \sum_{v \in \dsT_{n}} I_{r,v}$. Then, it should be clear that 
\begin{eqnarray} \label{eq32}
\scK(\dsT_{n}) \overset{d}{=} \sum_{1 \leq r \leq k} \scK_{r}(\dsT_{n}),
\end{eqnarray}

\noindent where $\overset{d}{=}$ denotes equal in distribution. 

Loosely speaking, we then consider the well-known {\sl depth-first} search walk or {\sl contour function}
$V_{n} = (V_{n}(t), t \in [0, 2(n-1)])$ of the (ordered) tree \(\dsT_{n}\) as depicted in
\Cref{fig:vn}, that is, $V_{n}(t)$ is ``the depth of the \(t\)-th vertex'' visited in this walk;
this will be made precise in the next section. As it is well-known (see Aldous \cite[Theorem 23 with
Remark 2]{Aldous1993} or \cite[Theorem 1]{Marck2003}), when $\mathbb{T}_{n}$ is a conditioned
Galton-Watson with offspring distribution satisfying \eqref{eq11}, we have that 
\begin{eqnarray*} 
(n^{-1/2} V_{n}(2(n-1)t), t \in [0,1]) \inlaw 2 \sigma^{-1} B^{\rm ex}, \hspace*{5mm} \text{as} \hspace*{2mm} n \rightarrow \infty. 
\end{eqnarray*}

\noindent in $C([0,1], \dsR_{+})$, with its usual topology, and where $B^{\rm ex} = (B^{\rm ex}(t),
t \in [0,1])$ is a standard normalized Brownian excursion. It has been shown in \cite[Lemma
2.1]{Cai2018} that\footnote{For two sequences of non-negative real numbers $(A_{n})_{n \geq 1}$ and
$(B_{n})_{n \geq 1}$ such that $B_{n} >0$, we write $A_{n} \sim B_{n}$ if $A_{n}/B_{n} \rightarrow
1$ as $n \rightarrow \infty$} \(\e{[I_{r,v}] \sim C_{r,k} d_{n}(v)^{-r/k}}\), for some (explicit)
constant \(C_{r,k} >0\), where $d_{n}(v)$ is the depth of the vertex $v \in \mathbb{T}_{n}$. Let
$\circ$ denote the root of $\mathbb{T}_{n}$. Thus, informally 
\begin{align*}\label{eq:appox}
    \E{\scK_{r}(\dsT_{n}) \given \dsT_{n}} 
    & \approx \sum_{v \in \mathbb{T}_{n} \setminus \{ \circ\}} \frac{C_{r,k}}{d_{n}(v)^{r/k}} 
    \approx   \frac{C_{r,k}}{2} \int_{0}^{2(n-1)}  \frac{{\rm d} t}{ V_{n}(t)^{r/k}} 
    \\
    &
    \approx \frac{C_{r,k}}{n^{-1+\frac{r}{2k}}} \int_{0}^{1} \left( \frac{V_{n}(2(n-1)t)}{\sqrt{n}} \right)^ {-\frac{r}{k}} \; \mathrm{d}t \\
    & \approx \frac{C_{r,k}}{n^{-1+\frac{r}{2k}}} \left( \frac{\sigma}{2} \right)^{\frac{r}{k}} \int_{0}^{1}  \frac{ \mathrm{d}t}{B^{\rm ex}(t)^{r/k}},
\end{align*}
 
\noindent when \(n\) is large. One then expects that
\begin{equation*}\label{eq:appox:1}
    \sigma^{-r/k} n^{-1+\frac{r}{2k}}  \E{\scK_{r}(\dsT_{n})} \sim C_{r,k} \E{\int_{0}^{1} \left( {2 B^{\rm ex}(t)} \right)^ {-r/k} \mathrm{d}t}, \hspace*{4mm} \text{as} \hspace*{2mm} n \rightarrow \infty,
\end{equation*}

\noindent which coincides with the right-hand side of  \eqref{eq:ZCRT} when \(r = q=1\). Note that
this informal computation suggests that\footnote{For two sequences of non-negative real numbers
$(A_{n})_{n \geq 1}$ and $(B_{n})_{n \geq 1}$ such that $B_{n}>0$, we write $A_{n} = O(B_{n})$ if
$\limsup_{n \rightarrow \infty } A_{n}/B_{n} < \infty$.} $\E{\scK_{r}(\dsT_{n})} =
O(n^{1-\frac{r}{2k}})$, for $r \in \{1, \dots, k\}$. As a consequence, Markov's inequality
implies that $ n^{-1+ \frac{1}{2k}} \scK_{r}(\dsT_{n}) \rightarrow 0$ in probability, as $n \rightarrow
\infty$, for $r \in \{2, \dots, k\}$. As shown later, by the identity in (\ref{eq32}), it would be enough to
prove \Cref{theo1} for $\scK_{1}(\dsT_{n})$ instead of $\scK(\dsT_{n})$. 

In the rest of the paper, \Cref{sec1} and \Cref{sec:thm1} make the above argument precise and extend
it to higher moments. This will allow us to use the method of moments for proving \Cref{theo1}. In
\Cref{sec4}, we also apply the same idea to get all moments of the number of records in paths and
several types of trees of logarithmic height, e.g., complete binary trees, split trees, uniform
random recursive trees and scale-free trees.

\begin{figure}
    \centering
    \begin{minipage}{0.25\textwidth}
        \centering
        \includegraphics[width=0.9\textwidth]{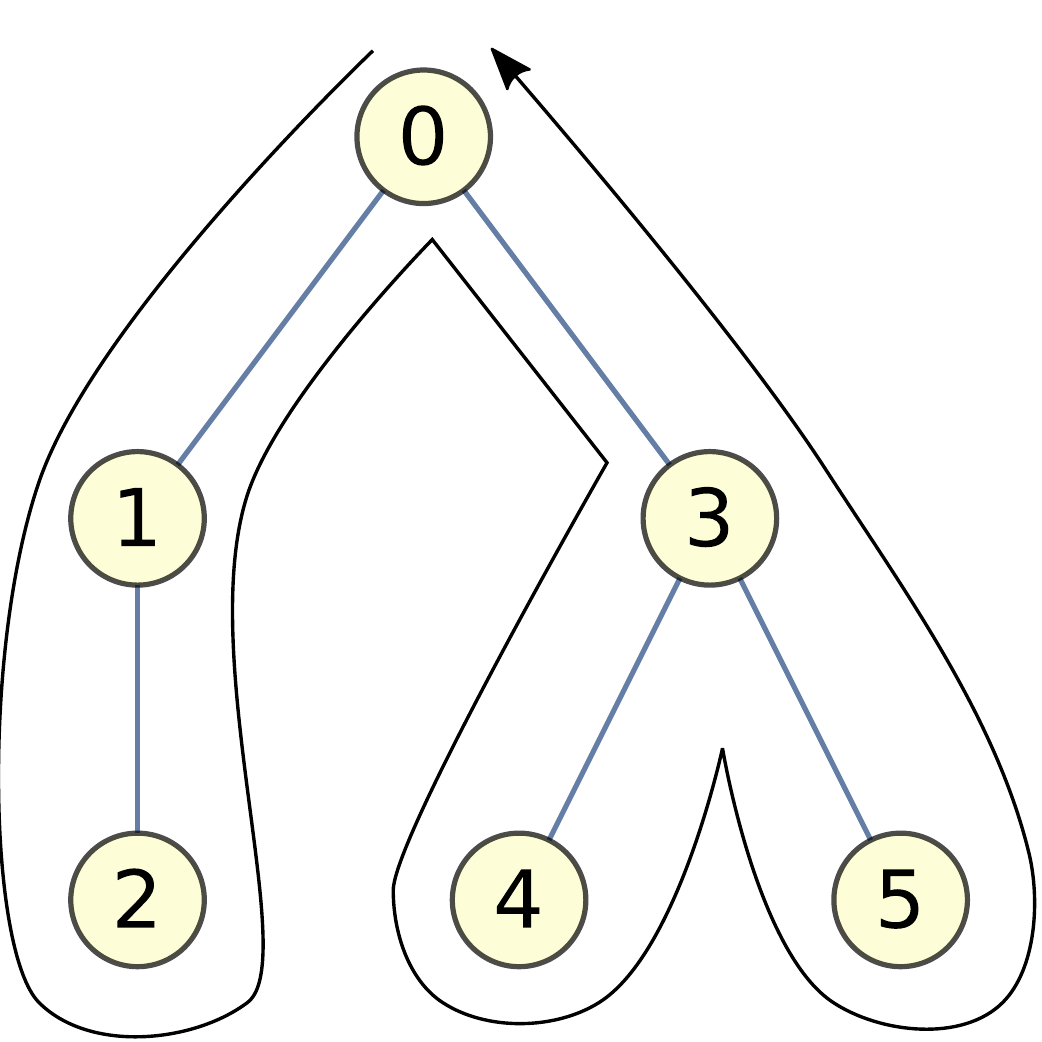} 
    \end{minipage}\hfill
    \begin{minipage}{0.70\textwidth}
        \centering
        \includegraphics[width=0.9\textwidth]{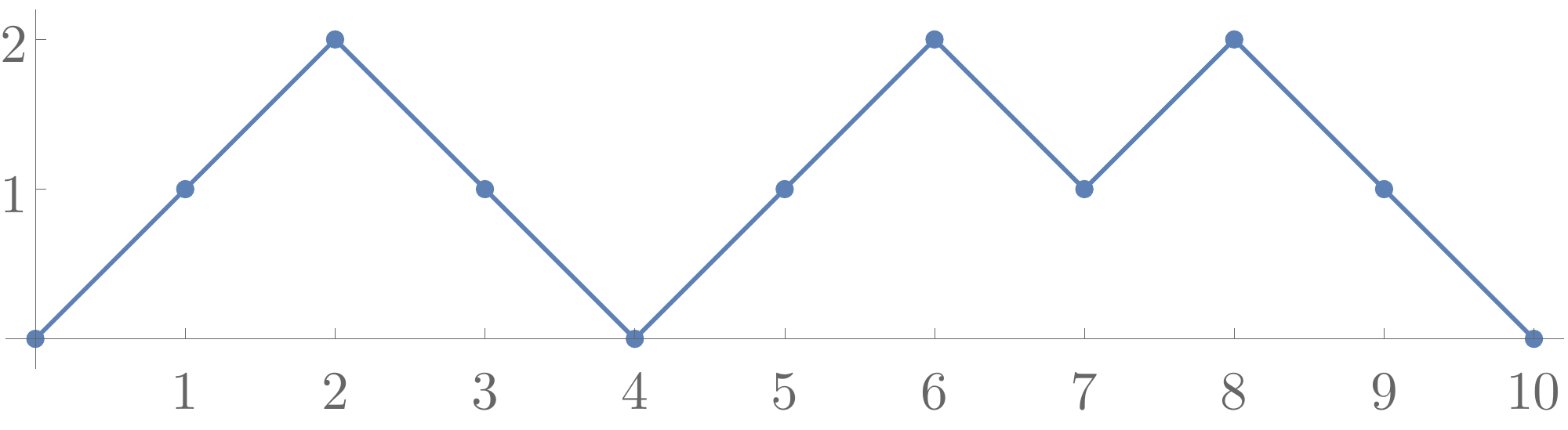} 
    \end{minipage}
    \caption{An example of a depth-first search walk in a tree and the corresponding
        \(V_{n}\).}
    \label{fig:vn}
\end{figure}

\section{Preliminary results}\label{sec1}

The purpose of this section is to establish a general convergence result for the number of
$1$-records $\scK_{1}(\dsT_{n})$ of a deterministic rooted ordered tree
$\dsT_{n}$. The results of this section can also be viewed as a generalization of those in
Janson \cite{Janson2006} and in Cai, et al.\ \cite{Cai2018}. Furthermore, these results will
allow us to study the convergence of $\scK(\dsT_{n})$ not only for
conditioned Galton-Watson trees, but also for other classes of random trees in Section \ref{sec4}. We
start by defining a probability measure through a continuous function in the same spirit as in
\cite[Theorem 1.9]{Janson2006}. Let $I \subseteq \dsR_{+}$ be an interval. For a function $f:
I \rightarrow \dsR_{+}$ and $t_{1}, \dots, t_{q} \in I$ with $q \in \dsN$, we define
\begin{eqnarray} \label{eq22}
L_{f}(t_{1}, \dots, t_{q}) \coloneqq \sum_{i =1}^{q} f(t_{(i)}) - \sum_{i=1}^{q-1} \inf_{t \in [t_{(i)}, t_{(i+1)}]} f(t),
\end{eqnarray} 

\noindent where $t_{(1)}, \dots, t_{(q)}$ are $t_{1}, \dots, t_{q}$ arranged in nondecreasing order. Notice that $L_{f}(t_{1}, \dots, t_{q})$ is symmetric in $t_{1}, \dots, t_{q}$ and that $L_{f}(t) = f(t)$ for $t \in I$. Define 
\begin{eqnarray} \label{eq6}
D_{f}(t_{1}) \coloneqq L_{f}(t_{1}) \hspace*{2mm} \text{and} \hspace*{2mm} D_{f}(t_{1}, \dots, t_{q}) \coloneqq   L_{f}(t_{1}, \dots, t_{q}) - L_{f}(t_{1}, \dots, t_{q-1}), \hspace*{2mm} \text{for} \hspace*{2mm} q \geq 2. 
\end{eqnarray}

\noindent We also consider the functional
\begin{eqnarray} \label{eq23}
G_{f}(\vect_{q}, \vecx_{q}) \coloneqq \exp \left( - \frac{D_{f}(t_{1})x_{1}^{k} + \dots + D_{f}(t_{1}, \dots, t_{q})x_{q}^{k} }{k !} \right), 
\end{eqnarray}

\noindent for  $\vecx_{q} = (x_{1}, \dots, x_{q}) \in \dsR_{+}^{q}$ and $\vect_{q} = (t_{1}, \dots, t_{q}) \in I^{q}$. If $I = [0,1]$, we further define, for $q \in \dsN$, 
\begin{eqnarray} \label{eq1}
m_{0}(f) \coloneqq 1 \hspace*{2mm} \text{and} \hspace*{2mm} m_{q}(f)  \coloneqq q ! \int_{0}^{1}
\int_{0}^{1}  \cdots \int_{0}^{1} \int_{0}^{\infty} \int_{0}^{x_{1}} \dots \int_{0}^{x_{q-1}}
G_{f}(\vect_{q}, \vecx_{q})\;{\rm d} \cev{\vecx}_{q}\;{\rm d} \cev{{\bf
        t}}_{q}, \quad \text{for } q \ge 2,
\end{eqnarray}

\noindent where $\cev{\vecx}_{q} = (x_{q}, \dots, x_{1})$ and $\cev{\vect}_{q} = (t_{q}, \dots, t_{1})$.

\begin{theorem} \label{Theo2}
    Let \(k \in \dsN\).
Suppose that $f \in C([0,1], \dsR_{+})$ is such that $\int_{0}^{1} f(t)^{-1/k}{\rm d} t  < \infty$. Then there exists a unique probability measure $\nu_{f}$ on $[0,\infty)$ with finite moments given by
\begin{eqnarray*}
\int_{[0, \infty)} x^{q} \nu_{f}({\rm d} x) = m_{q}(f), \hspace*{5mm} \text{for} \hspace*{2mm} q \in \dsZ_{\ge 0}. 
\end{eqnarray*}
\end{theorem}

\begin{proof}
We only prove uniqueness here. The proof for existence follows along the lines of \cite[Proof of
Theorem 1.9, Pages 18-19]{Janson2006} and details are left to the interested reader. Informally
speaking, the idea in \cite{Janson2006} for the proof of existence is to build a sequence of
functions that satisfy the conditions of \Cref{lemma3} below. Define the function
\begin{eqnarray} \label{eq3}
H_{f, q}(\vect_{q})  \coloneqq \int_{0}^{\infty} \int_{0}^{x_{1}} \dots \int_{0}^{x_{q-1}} G_{f}(\vect_{q}, \vecx_{q})\;{\rm d} \cev{\vecx}_{q}.
\end{eqnarray}

\noindent By changing the order of integration, we obtain that
\begin{eqnarray*}
H_{f, q}(\vect_{q}) =  \int_{0}^{\infty} \int_{x_{q}}^{\infty} \dots \int_{x_{2}}^{\infty}
G_{f}(\vect_{q}, \vecx_{q})\;{\rm d} \vecx_{q},
\end{eqnarray*}

\noindent for $\vecx_{q} = (x_{1}, \dots, x_{q}) \in \dsR_{+}^{q}$ and $\vect_{q} = (t_{1}, \dots, t_{q}) \in [0,1]^{q}$.
By making the change of variables $x_{q} = w_{q}, x_{q-1}= w_{q} + w_{q-1}, \dots, x_{1}= w_{q} + \cdots + w_{1}$, we see that
\begin{eqnarray*}
H_{f,q}(\vect_{q}) = 
\bigintss_{[0,\infty)^{q}} 
\exp
\left( 
    - \frac{1}{k!}
    \sum_{i=1}^{q}
    D_{f}(t_{1},\dots, t_{i})
    \left( 
    \sum_{j=i}^{q}
    w_{j}
    \right)^{k}
\right)
{\rm d} \vecw_{q},
\end{eqnarray*}

\noindent where $\vecw_{q} = (w_{1}, \dots, w_{q}) \in \dsR_{+}^{q}$. From the inequality $(x_{1}+\cdots + x_{q})^{k} \geq x_{1}^{k} + \cdots + x_{q}^{k}$, we observe that
%
%
\begin{eqnarray}
H_{f,q}(\vect_{q}) 
& \leq &
\prod_{j=1}^{q} \int_{0}^{\infty} 
\exp\left( 
    - \frac{w_{j}^{k}}{k!} \sum_{i=1}^{j} D_{f}(t_{1},\dots,t_{i})
\right)
\, \mathrm{d}w_{j} \nonumber
\\
& = & \Gamma \left(1 + 1/k \right)^{q} \Gamma(1+k)^{q/k} \prod_{j=1}^{q}  \left( \sum_{i=1}^{j} D_{f}(t_{1}, \dots, t_{i}) \right)^{-1/k} \nonumber \\
& = & \Gamma \left(1 + 1/k \right)^{q} \Gamma(1+k)^{q/k}  \prod_{i=1}^{q}   L_{f}(t_{1}, \dots, t_{i})^{-1/k} \nonumber \\
& \leq & \Gamma \left(1 + 1/k \right)^{q} \Gamma(1+k)^{q/k} \prod_{i=1}^{q}  f(t_{i})^{-1/k}, \label{eq5}
\end{eqnarray}

\noindent where for the last inequality we have used the fact that $L_{f}(t_{1}, \dots, t_{i}) \geq
\max_{1 \leq j \leq i} f(t_{j})$, for $1 \leq i \leq q$. The later follows from the symmetry of
$L_{f}$; see \cite[Lemma 4.1]{Janson2006} for a proof. Then, the previous inequality allows us to
conclude that  
\begin{eqnarray*}
0 \leq m_{q}(f) \leq q ! \, \Gamma \left(1 + 1/k \right)^{q} \Gamma(1+k)^{q/k} \left( \int_{0}^{1} f(t)^{-1/k}{\rm d} t  \right)^{q}.
\end{eqnarray*}

\noindent We conclude that there exists $a >0$ such that $\sum_{q=0}^{\infty}
m_{q}(f) \frac{x^{q}}{q !} < \infty$, for $0 \leq x < a$. Then a probability measure with moments
$m_{q}(f)$ has a finite generating function in a neighbourhood of $0$. Thus, it is well-known that this
implies that the probability measure is unique; see, e.g., \cite[Section 4.10]{Gut2013}. 
\end{proof}

Consider a rooted ordered tree $\dsT_{n}$ with root $\circ$ and $n \in \dsN$ vertices.
We now explain how $\dsT_{n}$ can be encoded by a continuous function. We define the so-called
{\sl depth-first search function} \cite[page 260]{Aldous19912}, $\psi_{n}:\{0,1, \dots, 2(n-1)\}
\rightarrow \{\, \text{vertices of} \, \, \dsT_{n} \}$ such that \(\psi_{n}(i)\) is the
\((i+1)\)-th vertex visited in a depth-first walk on the tree starting from the root \(\circ\). Note that
$\psi_{n}(i)$ and $\psi_{n}(i+1)$ always are neighbours, and thus, we extend $\psi$ to $[0,2(n-1)]$
by letting, for $1 \leq i < t < i+1 \leq 2(n-1)$, $\psi_{n}(t)$ to be the one of $\psi_{n}(i)$ and
$\psi_{n}(i+1)$ that has largest depth (recall that the depth of a vertex $v \in \dsT_{n}$ is the
distance, i.e., number of edges, between $\circ$ to $v$). Let $d_{n}(v)$ be the depth of a vertex $v
\in \dsT_{n}$. We further define the {\sl depth-first walk} $V_{n}$ of $\dsT_{n}$ by 
\begin{eqnarray*}
V_{n}(i) \coloneqq d_{n}(\psi(i)), \hspace*{5mm} i \in \{0, \dots, 2(n-1)\},
\end{eqnarray*}

\noindent and extend $V_{n}$ to $[0,2(n-1)]$ by linear interpolation. Thus $V_{n}
\in C([0,2(n-1)], \dsR_{+})$. See \Cref{fig:vn} for an example of \(V_{n}\). Furthermore, we normalize the domain of $V_{n}$ to $[0, 1]$ by
defining 
\begin{eqnarray} \label{eq35}
\widetilde{V}_{n}(t) \coloneqq V_{n}(2(n-1)t) \hspace*{4mm} \text{and} \hspace*{4mm} \widehat{V}_{n}(t) \coloneqq \lceil V_{n}(2(n-1)t) \rceil,
\end{eqnarray}

\noindent for $t \in [0,1]$. Thus $\widetilde{V}_{n} \in C([0,1], \dsR_{+})$. Note that $d_{n}(\psi(t)) = \lceil V_{n}(t) \rceil$, for $t \in [0,2(n-1)]$. Moreover,
\begin{eqnarray} \label{eq10}
\max_{v \in \dsT_{n}} d_{n}(v) = \sup_{t \in [0, 2(n-1)]} V_{n}(t) = \sup_{t \in [0, 1]} \widetilde{V}_{n}(t). 
\end{eqnarray}

We now state the central result of this section, that is, a general limit theorem in distribution for the number of $1$-records $\scK_{1}(\dsT_{n})$ of a deterministic rooted tree $\dsT_{n}$ with $n$ vertices. It is important to notice that  $\scK_{1}(\dsT_{n})$ is a random variable since the $1$-records are random. From now on, we always assume that $k \geq 2$. 

\begin{lemma} \label{lemma3}
Suppose that $(\dsT_{n})_{n \geq 1}$ is a sequence of ordered (deterministic) rooted trees, and denote the corresponding normalized depth-first walks by $\widetilde{V}_{n}$ and $\widehat{V}_{n}$. Suppose that there exists a sequence $(a_{n})_{n \geq 1}$ of non-negative real numbers with $\lim_{n \rightarrow \infty} a_{n} = 0$, $\lim_{n \rightarrow \infty} na_{n}^{1/k} = \infty$ and a function $f \in C([0,1], \dsR_{+})$ such that
\begin{itemize}
\item[(a)] $ \displaystyle a_{n} \widetilde{V}_{n}(t) \rightarrow f(t)$, in $C([0,1], \dsR_{+})$, as $n \rightarrow \infty$. 

\item[(b)] $ \displaystyle \int_{0}^{1} ( a_{n}\widehat{V}_{n}(t) )^{-1/k}\;{\rm d}t  \rightarrow \int_{0}^{1} f(t)^{-1/k}\;{\rm d}t < \infty$, as $n \rightarrow \infty$. 
\end{itemize}

\noindent Then, for each $q \in \dsZ_{\ge 0}$, 
\[
    n^{-q}a_{n}^{-q/k}  \dsE[\scK_{1}(\dsT_{n})^{q}] 
    \rightarrow m_{q}(f),
\]
as $n \rightarrow \infty$, where $m_{q}(f)$ is defined in \eqref{eq1}. Moreover, $n^{-1} a_{n}^{-1/k} \scK_{1}(\dsT_{n}) \inlaw Z_{f}$, as $n \rightarrow \infty$, where $Z_{f}$ is a random variable with distribution $\nu_{f}$ defined by \Cref{Theo2}. 
\end{lemma}

Before proving \Cref{lemma3}, we need to establish some preliminary results
and to introduce some further notation. For $q \in \dsN$ and vertices $v_{1}, \dots, v_{q} \in \dsT_{n}$, let
$L_{n}(v_{1}, \dots, v_{q})$ be the number of edges in the subtree of $\dsT_{n}$ spanned by $v_{1},
\dots, v_{q}$ and its root $\circ$ (i.e., the minimal number of edges that are needed to connect $v_{1}, \dots, v_{q}$ and $\circ$). We write $D_{n}(v_{1}) \coloneqq L_{n}(v_{1})$ and $D_{n}(v_{1},
\dots, v_{q}) \coloneqq L_{n}(v_{1}, \dots, v_{q}) - L_{n}(v_{1}, \dots, v_{q-1})$ for $q \geq 2$.
We also consider the functional
\begin{eqnarray} \label{eq29}
G_{n}(\vecv_{q}, \vecx_{q}) \coloneqq \exp \left( - \frac{D_{n}(v_{1})x_{1}^{k} + \dots + D_{n}(v_{1}, \dots, v_{q})x_{q}^{k} }{k !} \right), 
\end{eqnarray}

\noindent for $\vecx_{q} = (x_{1}, \dots, x_{q}) \in \dsR_{+}^{q}$ and $\vecv_{q} = (v_{1}, \dots,
v_{q}) \in \dsT_{n}^{q}$. We denote by $\Gamma(k, \cdot)$ the upper incomplete gamma function of
parameter $k \in \dsN$, i.e.,
\begin{equation*}\label{MHVXZ}
 \Gamma(k,x) = \int_{x}^{\infty} t^{k-1} e^{-t} \, \mathrm{d}t, \hspace*{4mm} \text{for} \hspace*{2mm} x \geq 0.  
\end{equation*}

\begin{remark} \label{remark1}
Let $\dsT_{n}$ be an ordered (deterministic) rooted tree with depth-first search walk $\psi_{n}$ and the corresponding function $V_{n}$. It is not difficult to see that $L_{n}$ and $L_{\lceil V_{n} \rceil}$ are connected, in
the sense that $L_{n}(\psi_{n}(t_{1}), \dots, \psi_{n}(t_{q})) = L_{\lceil V_{n} \rceil}(t_{1},
\dots, t_{q})$ for $t_{1}, \dots, t_{q} \in [0,2(n-1)]$; see \cite[Lemma 4.4]{Janson2006} for a
proof of this fact.  
\end{remark}

\begin{lemma} \label{lemma2}
Let $\dsT_{n}$ be an ordered (deterministic) rooted tree with $n \in \dsN$ vertices. Suppose that there
exists a sequence $(a_{n})_{n \geq 1}$ of non-negative real numbers such that $\lim_{n \rightarrow \infty} a_{n} = 0$ and \(\max_{v \in\dsT_{n}} d_{n}(v) = O(a_{n}^{-1})\). Let $\alpha
\coloneqq \frac{1}{2}\left(\frac{1}{k} + \frac{1}{k+1} \right)$ and $x_{0}  \coloneqq
a_{n}^{\alpha}$. Then, for $q \in \dsN$ and uniformly for all $x \in [0, x_{0}]$,
\begin{eqnarray*}
\dsP({\rm Gamma}(k) > x)^{D_{n}(v_{1}, \dots, v_{q})} = \left( \frac{\Gamma(k,x)}{\Gamma(k)} \right)^{D_{n}(v_{1}, \dots, v_{q})} = (1 + O ( a_{n}^{\frac{1}{2k}} ) ) \exp \left( - \frac{D_{n}(v_{1}, \dots, v_{q})x^{k}}{k!} \right),
\end{eqnarray*}

\noindent where the vertices $v_{1}, \dots, v_{q} \in \dsT_{n}$.
\end{lemma}

\begin{proof}
    Our claim can be shown along the lines of \cite[Proof of Lemma 5.1]{Cai2018}. 
\end{proof}

Recall that for two sequences of non-negative real numbers $(A_{n})_{n \geq 1}$ and $(B_{n})_{n \geq
1}$ such that $B_{n} >0$, one writes $A_{n}=o(B_{n})$ if $\lim_{n \rightarrow \infty} A_{n}/B_{n} =
0$.

\begin{lemma} \label{lemma1}
Let $\dsT_{n}$ be an ordered (deterministic) rooted tree with $n \in \dsN$ vertices. Suppose that there
exists a sequence $(a_{n})_{n \geq 1}$ of non-negative real numbers with $\lim_{n \rightarrow \infty}
a_{n} = 0$, $\lim_{n \rightarrow \infty} na_{n}^{1/k} = \infty$ and \(\max_{v \in \dsT_{n}} d_{n}(v) = O(a_{n}^{-1})\). Then the moments of $\scK_{1}(\dsT_{n})$
are given by 
\begin{eqnarray*}
n^{-q} a_{n}^{-q/k}  \dsE[\scK_{1}(\dsT_{n})^{q}] = 
(1 + O ( a_{n}^{\frac{q}{2k}} ) )  q! \int_{0}^{1}  \cdots \int_{0}^{1} \bar{H}_{n,q}(\vect_{q})\;{\rm d} \cev{\vect}_{q} + o(1),
\end{eqnarray*}

\noindent where 
\begin{eqnarray} \label{eq7}
\bar{H}_{n,q}(\vect_{q}) \coloneqq \int_{0}^{\infty} \int_{0}^{x_{1}} \cdots \int_{0}^{x_{q-1}}
G_{a_{n}\widehat{V}_{n}}(\vect_{q}, \vecx_{q})  \exp\left({- a_{n}^{1/k}\sum_{i=1}^{q}x_{i}}\right)\;{\rm d} \cev{\vecx}_{q}, \hspace*{4mm} \text{for} \hspace*{2mm} q \in \dsN. 
\end{eqnarray}
\end{lemma}

\begin{proof}
For simplicity, we write $X_{q} \coloneqq \scK_{1}(\dsT_{n})^{q}$ for $q \in \dsZ_{\ge 0}$ and note that $X_{q} = X_{1}^{q}$. For $q \in \dsN$, we observe that
\begin{eqnarray*}
X_{q} = ( X_{1} -1+1)^{q} = (X_{1}-1)^{q} + \sum_{p=0}^{q-1} \binom{q}{p} (X_{1}-1)^{p} =  (X_{1}-1)^{q} + Y_{q}.
\end{eqnarray*}

\noindent where $Y_{q} \coloneqq \sum_{p=0}^{q-1} \sum_{l=0}^{p}  \binom{q}{p} \binom{p}{l} (-1)^{p-l} X_{l}$. Recall that $I_{1,v}$ is the indicator that $v \in \dsT_{n}$ is a $1$-record defined in (\ref{eq2}). By the previous identity, we have that 
\begin{eqnarray*}
X_{q} = \sum_{v_{1}, \dots, v_{q} \in \dsT_{n} \setminus \{\circ \} } I_{1, v_{1}} \cdots I_{1, v_{q}} + Y_{q}= q! \sum_{v_{1}, \dots, v_{q} \in \dsT_{n} \setminus \{\circ \}  } \llbracket \scE(v_{1}, \dots, v_{q}) \rrbracket + Y_{q}
\end{eqnarray*}

\noindent where $\scE(v_{1},
\dots, v_{q}) \coloneqq \{ E_{1,v_{q}} < \cdots < E_{1,v_{1}} \, \, \text{and} \, \, v_{1}, \dots,
v_{q} \, \, \text{are all $1$-records} \}$; recall that $E_{1,v_{1}}, \dots, E_{1,v_{q}}$ are
independent random variables with an $\text{Exp}(1)$ distribution. To see the last identity, note
that each product $I_{1, v_{1}} \cdots I_{1, v_{q}}$ occurs $q!$ times with indices permuted and for
exactly one of these permutations we have that $E_{1,v_{q}} < \cdots < E_{1,v_{1}}$. 

Consider the simple case \(q=2\).  Conditioning on \(E_{1,v_{2}} = x_{2} < E_{1,v_{1}} =
    x_{1}\), we see that \(v_{1}\) and \(v_{2}\) are both \(1\)-records, if and only if, the following two events happen:
\begin{itemize}
    \item[(i)] the \(D_{n}(v_{1})\) ancestors of \(v_{1}\) are removed after time \(x_{1}\);
    \item[(ii)] the \(D_{n}(v_{1}, v_{2})\) vertices which are ancestors of \(v_{2}\) but not of \(v_{1}\) are removed after time \(x_{2}\).
\end{itemize}

\noindent Since \(x_{2} < x_{1}\), we note that the event (i) implies that the vertices which are both the ancestors of
\(v_{1}\) and \(v_{2}\) are removed after \(x_{1}\). 
Let $g(x) \coloneqq \dsP(\text{Gamma}(k) > x)$ for $x \in \mathbb{R}_{+}$. Since the events (i) and (ii) are independent,  we have
\begin{equation}\label{eq:scE:2}
    \p{\scE(v_{1},v_{2})} = 
    \int_{0}^{\infty}
    \int_{0}^{x_{1}}
    g(x_{1})^{D_{n}(v_{1})} g(x_{2})^{D_{n}(v_{1}, v_{2})} e^{-x_{1} - x_{2}} 
    \;
    \mathrm{d}x_{2}
    \;
    \mathrm{d}x_{1}.
\end{equation}

\noindent Recall that we are assuming \(k \ge 2\). Otherwise, when \(k=1\), the above equality is not entirely correct since \(\scE(v_{1}, v_{2})\) is impossible if \(v_{2}\) is
an ancestor of \(v_{1}\); see \cite[Lemma 4.3]{Janson2006} for details in the case $k=1$.

By generalizing the previous argument to \(q \in \dsN\), we see that
\begin{align*}
\dsP(\scE(v_{1}, \dots, v_{q})) & = \int_{0}^{\infty} \int_{0}^{x_{1}} \cdots \int_{0}^{x_{q-1}} g(x_{1})^{D_{n}(v_{1})} g(x_{2})^{D_{n}(v_{1}, v_{2})} \cdots g(x_{q})^{D_{n}(v_{1}, \dots, v_{q})} e^{- \sum_{i=1}^{q}x_{i}}\;{\rm d} \cev{\vecx}_{q} \\
& = \int_{0}^{x_{0}} \int_{0}^{x_{1}} \cdots \int_{0}^{x_{q-1}} g(x_{1})^{D_{n}(v_{1})} g(x_{2})^{D_{n}(v_{1}, v_{2})} \cdots g(x_{q})^{D_{n}(v_{1}, \dots, v_{q})} e^{- \sum_{i=1}^{q}x_{i}}\;{\rm d} \cev{\vecx}_{q} \\
& \hspace*{5mm} +\int_{x_{0}}^{\infty} \int_{0}^{x_{1}} \cdots \int_{0}^{x_{q-1}} g(x_{1})^{D_{n}(v_{1})} g(x_{2})^{D_{n}(v_{1}, v_{2})} \cdots g(x_{q})^{D_{n}(v_{1}, \dots, v_{q})} e^{- \sum_{i=1}^{q}x_{i}}\;{\rm d} \cev{\vecx}_{q} \\
& = A_{1} + A_{2},
\end{align*}

\noindent where $\cev{\vecx}_{q} = (x_{q}, \dots, x_{1}) \in \dsR_{+}^{q}$, $x_{0} = a_{n}^{\alpha}$ and $\alpha = \frac{1}{2}  \left(\frac{1}{k} + \frac{1}{k+1} \right)$. On the one hand, \Cref{lemma2} implies that
\begin{eqnarray*}
A_{2} \leq \int_{x_{0}}^{\infty} g(x)^{D_{n}(v_{1})} e^{-x}\;{\rm d} x \leq g(x_{0})^{D_{n}(v_{1})} =
O\left( \exp\left({-\frac{ x_{0}^{k}}{2a_{n}k!}}  \right) \right);
\end{eqnarray*}

\noindent we have used our assumption \(\max_{v \in \dsT_{n}} d_{n}(v) = O(a_{n}^{-1})\). On the other hand, \Cref{lemma2} also implies that 
\begin{eqnarray*}
A_{1} & = & (1 + O ( a_{n}^{\frac{1}{2k}} ) )^{q} \int_{0}^{x_{0}} \int_{0}^{x_{1}} \cdots \int_{0}^{x_{q-1}} G_{n}(\vecv_{q}, \vecx_{q})  e^{- \sum_{i=1}^{q}x_{i}}\;{\rm d} \cev{\vecx}_{q}\\
& = & (1 + O ( a_{n}^{\frac{q}{2k}} ) ) \int_{0}^{\infty} \int_{0}^{x_{1}} \cdots \int_{0}^{x_{q-1}} G_{n}(\vecv_{q}, \vecx_{q})  e^{- \sum_{i=1}^{q}x_{i}}\;{\rm d} \cev{\vecx}_{q} + A_{3},
\end{eqnarray*}

\noindent where $\vecv_{q} = (v_{1}, \dots, v_{q}) \in \dsT_{n}^{q}$ and
\begin{eqnarray*}
A_{3} = (1 + O ( a_{n}^{\frac{q}{2k}} ) ) \int_{x_{0}}^{\infty} \int_{0}^{x_{1}} \cdots
\int_{0}^{x_{q-1}} G_{n}(\vecv_{q}, \vecx_{q})  e^{- \sum_{i=1}^{q}x_{i}}\;{\rm d}
\cev{\vecx}_{q}= O\left( \exp\left({-\frac{ x_{0}^{k}}{2 a_{n} k!}}\right) \right);
\end{eqnarray*}

\noindent this estimation can be deduced similarly as the one for the integral $A_{2}$.  Therefore, the previous estimations and Remark \ref{remark1} allow us to conclude that 
\begin{align} \label{eq15}
\dsE[X_{q}] &  =  (1 + O( a_{n}^{\frac{q}{2k}} ) )  q! \sum_{v_{1}, \dots, v_{q} \in \dsT_{n} \setminus \{ \circ \} }  \int_{0}^{\infty} \int_{0}^{x_{1}} \cdots \int_{0}^{x_{q-1}} G_{n}(\vecv_{q}, { \bf x}_{q})  e^{- \sum_{i=1}^{q}x_{i}}\;{\rm d} \cev{\vecx}_{q}   \nonumber \\
& \hspace*{80mm }+ \dsE[Y_{q}] + o(n^{q}a_{n}^{q/k})\\ 
& =  (1 + O( a_{n}^{\frac{q}{2k}} ) ) q! 2^{-q}  \int_{0}^{2(n-1)} \cdots \int_{0}^{2(n-1)}  \int_{0}^{\infty} \int_{0}^{x_{1}} \cdots \int_{0}^{x_{q-1}} G_{\lceil V_{n} \rceil }(\vect_{q}, \vecx_{q}) e^{- \sum_{i=1}^{q}x_{i}}\;{\rm d} \cev{\vecx}_{q}\;{\rm d} \cev{\vect}_{q}  \nonumber \\
& \hspace*{80mm }+ \dsE[Y_{q}] + o(n^{q}a_{n}^{q/k}) \nonumber \\ 
&  =  (1 + O( a_{n}^{\frac{q}{2k}} ) ) q! n^{q} \int_{0}^{1} \cdots \int_{0}^{1}  \int_{0}^{\infty} \int_{0}^{x_{1}} \cdots \int_{0}^{x_{q-1}} G_{\widehat{V}_{n}}(\vect_{q}, \vecx_{q})  e^{- \sum_{i=1}^{q}x_{i}}\;{\rm d} \cev{\vecx}_{q}\;{\rm d} \cev{\vect}_{q}+  \nonumber \\
& \hspace*{80mm }+ \dsE[Y_{q}] + o(n^{q}a_{n}^{q/k}); \nonumber
\end{align}

\noindent note that if we had not excluded the root, we would not be able to write the sum as an integral.
By making the change of variables $x_{i} = a_{n}^{1/k} w_{i}$, for $1 \leq i \leq q$, we have that 
\begin{eqnarray*} 
\dsE[X_{q}] = (1 + O( a_{n}^{\frac{q}{2k}} ) ) q! n^{q} a_{n}^{q/k} \int_{0}^{1} \cdots \int_{0}^{1} \bar{H}_{n,q}(\vect_{q})  {\rm d} \cev{\vect}_{q}+ \dsE[Y_{q}] + + o(n^{q}a_{n}^{q/k}). \nonumber
\end{eqnarray*}

\noindent Finally, our claim follows by induction on $q \in \dsN$ and the assumption $\lim_{n \rightarrow \infty} na_{n}^{1/k} = \infty$.
\end{proof}


We are now able to establish \Cref{lemma3}.

\begin{proof}[Proof of \Cref{lemma3}]
    First note that by condition (a) of \Cref{lemma3} and \eqref{eq10}, we have \(\max_{v \in
\dsT_{n}} d_{n}(v) = \sup_{t \in [0,1]} \widetilde{V}_{n}(t) = O(a_{n}^{-1})\). Thus
the conditions for \Cref{lemma2} and \Cref{lemma1} are satisfied.

Recall the functions $\bar{H}_{n,q}$ and $H_{f,q}$ defined in (\ref{eq7}) and (\ref{eq3}), respectively. Therefore, notice that we only need to show that 
\begin{eqnarray} \label{eq8}
    \int_{[0,1]^{q}} \bar{H}_{n,q}(\vect_{q})\;{\rm d} \cev{\vect}_{q} \rightarrow  
    \int_{[0,1]^{q}} H_{f,q}(\vect_{q})\;{\rm d} \cev{\vect}_{q}, \hspace*{4mm} \text{as} \hspace*{2mm} n \rightarrow \infty.
\end{eqnarray}
 
\noindent The above convergence together with Lemma \ref{lemma1} implies that $\dsE[\scK_{1}(\dsT_{n})^{q}] = O(n^{q}a_{n}^{q/k})$ which clearly proves the first claim in Lemma \ref{lemma3}. The second claim follows immediately from \Cref{Theo2} and the method of moments. 

We henceforth prove the claim in \eqref{eq8}. Recall that a sequence $(g_{n})_{n \geq 1}$ of
non-negative functions on a measure space $(\Omega, \scF, \mu)$ with total mass $1$, i.e.,
$\mu(\Omega) =1$, is uniformly integrable if $\int_{\Omega} g_{n}\;{\rm d} \mu < \infty$ for all
\(n \ge 1\) and 
\begin{eqnarray*}
\sup_{A \in \scF: \mu(A) \leq \delta} \sup_{n \geq 1} \int_{A} g_{n}\;{\rm d} \mu \rightarrow 0, \hspace*{4mm} \text{as} \hspace*{2mm} \delta \rightarrow 0. 
\end{eqnarray*}

\noindent We also recall the following useful result on uniformly integrable sequences of functions. Suppose further that $g_{n} \rightarrow g$ almost everywhere as $n \rightarrow \infty$.  By \cite[Proposition 4.12]{Ka2002}, we know that
\begin{eqnarray} \label{eq9}
(g_{n})_{n \geq 1} \hspace*{3mm} \text{is uniformly integrable if and only if} \hspace*{3mm} \int g_{n}\;{\rm d} \mu \rightarrow \int g\;{\rm d} \mu < \infty, \hspace*{2mm} \text{as} \hspace*{2mm} n \rightarrow \infty.
\end{eqnarray} 

\noindent Then in order to prove (\ref{eq8}), it is enough to check the following:
\begin{itemize}
\item[(i)] The sequence $(\bar{H}_{n,q})_{n \geq 1}$ is uniformly integrable on $[0,1]^{q}$, and
\item[(ii)] $\bar{H}_{n,q} \rightarrow H_{f,q}$ as $n \rightarrow \infty$. 
\end{itemize}
 
\noindent We start by showing (i). Note that $|a_{n}\widetilde{V}_{n}(t) - a_{n}\widehat{V}_{n}(t)| \leq a_{n}$ for $t \in [0,1]$. Then, the assumption (a) implies that $a_{n}\widehat{V}_{n}(t) \rightarrow f(t)$ and $1/(a_{n}\widehat{V}_{n}(t))^{1/k} \rightarrow (1/f(t))^{1/k}$, for every $t \in [0,1]$, as $n \rightarrow \infty$. Moreover, the assumption (b) shows that $(1/(a_{n}\widehat{V}_{n}(t))^{1/k})_{n \geq 1}$ is uniformly integrable on $[0,1]$. More generally, for every fixed $q \in \dsN$ and $\vect_{q} = (t_{1}, \dots, t_{q})$, define the function $ \widetilde{H}_{n,q}(\vect_{q}) \coloneqq  (a_{n}\widehat{V}_{n}(t_{1}) \cdots a_{n}\widehat{V}_{n}(t_{q}) )^{-1/k}$. We then observe that
\begin{eqnarray*}
\int_{0}^{1} \cdots \int_{0}^{1}  \widetilde{H}_{n,q}(\vect_{q})\;{\rm d} \vect_{q}  =  \left(\int_{0}^{1} \left( a_{n}\widehat{V}_{n}(t) \right)^{ -1/k}\;{\rm d} t \right)^{q} 
& \rightarrow & \left(\int_{0}^{1} f(t)^{-1/k}\;{\rm d} t \right)^{q} \\
& = & \int_{0}^{1} \cdots \int_{0}^{1}\left( f(t_{1}) \cdots f(t_{q}) \right)^{-1/q}\;{\rm d} \vect_{q},
\end{eqnarray*}

\noindent as $n \rightarrow \infty$. Thus the result in (\ref{eq9}) shows that the sequence
$(\widetilde{H}_{n,q})_{n \geq 1}$ is uniformly integrable on $[0,1]^{q}$. Next notice that the
inequality $\exp(- a_{n}^{1/k}(x_{1} + \cdots + x_{q} )) \leq 1$ implies that
$\bar{H}_{n,q}(\vect_{q}) \leq H_{a_{n}\widehat{V}_{n}, q}(\vect_{q})$, where
$H_{a_{n}\widehat{V}_{n}, q}$ is defined in (\ref{eq3}). Then the inequality (\ref{eq5}) implies that
there exists a constant $C_{k,q} >0$ such that $\bar{H}_{n,q}(\vect_{q}) \leq C_{k,q}
\widetilde{H}_{n,q}(\vect_{q})$. Hence (i) follows by applying \cite[Theorem 4.5]{Gut2013}. 

Finally, we verify (ii). Recall that condition (a) implies that $a_{n}\widehat{V}_{n}(t) \rightarrow f(t)$, for every $t \in [0,1]$, as $n \rightarrow \infty$. Hence, whenever $0 \leq t_{1} \leq t_{2} \leq 1$, $\inf_{t \in [t_{1}, t_{2}]} a_{n}\widehat{V}_{n}(t) \rightarrow \inf_{t \in [t_{1}, t_{2}]} f(t)$ as $n \rightarrow \infty$. Thus, for $q \in \dsN$, the equation (\ref{eq6}), implies that $D_{a_{n}\widehat{V}_{n}}(t_{1}, \dots, t_{q}) \rightarrow D_{f}(t_{1}, \dots, t_{q})$ uniformly for $t_{1}, \dots, t_{q} \in [0,1]$ as $n \rightarrow \infty$. Then, for $\vecx_{q} \in \dsR_{+}^{q}$ and $\vect_{q} \in [0,1]^{q}$,
\begin{eqnarray*}
G_{a_{n}\widehat{V}_{n}}(\vect_{q}, \vecx_{q})  e^{- a_{n}^{1/k}\sum_{i=1}^{q}x_{i}} \rightarrow G_{f}(\vect_{q}, \vecx_{q}), \hspace*{3mm} \text{as} \hspace*{2mm} n \rightarrow \infty.
\end{eqnarray*}

\noindent Note that for $\varepsilon \in (0,1)$ there exists $N \in \dsN$ such that 
\begin{eqnarray*}
G_{a_{n}\widehat{V}_{n}}(\vect_{q}, \vecx_{q})  e^{- a_{n}^{1/k}\sum_{i=1}^{q}x_{i}} \leq
\exp\left({-\frac{(1-\varepsilon)f(t_{1})x_{1}^{1/k}}{k!}}  \right), \hspace*{3mm} \text{for} \hspace*{2mm} n \geq N. 
\end{eqnarray*}
 
\noindent Moreover, note that condition (b) implies that the function on the right-hand side of the inequality is integrable on $\{ \vecx_{q} \in \dsR_{+}: 0 \leq x_{q} \leq \cdots \leq x_{1} < \infty \}$. Therefore, it should be clear that (ii) follows by the dominated convergence theorem. This finishes the proof.  
\end{proof}

We can apply similar ideas as in the proofs of \Cref{lemma3} and \Cref{lemma1} to estimate
the mean of the number of $r$-records $\scK_{r}(\dsT_{n})$. It is important to mention  that we have
not tried to estimate higher moments of $\scK_{r}(\dsT_{n})$ to obtain a limit theorem in
distribution for this quantity. We believe that our methods can be used but the computations will be
more involved and we decided not to do it. Furthermore, the next results show that
$\scK_{r}(\dsT_{n})$ is of smaller order than $\scK_{1}(\dsT_{n})$ and hence it will not contribute
(in the limit) to the distribution of the $k$-cut number $\scK(\dsT_{n})$. 

\begin{lemma} \label{lemma7}
Let $\dsT_{n}$ be an ordered (deterministic) rooted tree with $n \in \dsN$ vertices. Suppose that
there exists a sequence $(a_{n})_{n \geq 1}$ of non-negative real numbers with $\lim_{n \rightarrow
\infty} a_{n} = 0$, $\lim_{n \rightarrow \infty} na_{n}^{1/k} = \infty$ and \(\max_{v \in \dsT_{n}}
d_{n}(v) = O(a_{n}^{-1})\). Then, for $r \in \{1, \dots, k\}$,  
\begin{eqnarray*}
n^{-1} a_{n}^{-r/k}  \dsE[\scK_{r}(\dsT_{n})] = (1+O(a_{n}^{\frac{1}{2k}}))  \int_{0}^{1} \int_{0}^{\infty} \frac{x^{r-1} e^{-a_{n}^{1/k}x}}{\Gamma(r)} e^{-\frac{a_{n} \widehat{V}_{n}(t)  x^{k}}{k!}}\;{\rm d} x {\rm d} t + o(1).
\end{eqnarray*}
\end{lemma}

\begin{proof}
Note that the case $r =1$ has been proven in \Cref{lemma1}. We follow a similar strategy to prove the case $r \in \{2, \dots, k\}$. Recall that $I_{r,v}$ is the indicator of the event that the vertex $v \in \dsT_{n}$ is an $r$-record defined in (\ref{eq2}). We observe that
\begin{eqnarray*}
\dsE[I_{r,v}] & = & \int_{0}^{\infty} \frac{x^{r-1}e^{-x}}{\Gamma(r)} \dsP({\rm Gamma(k)} > x)^{d_{n}(v)}\;{\rm d}x \\
& = & \int_{0}^{x_{0}} \frac{x^{r-1}e^{-x}}{\Gamma(r)} \dsP({\rm Gamma(k)} > x)^{d_{n}(v)}\;{\rm d}x + \int_{x_{0}}^{\infty} \frac{x^{r-1}e^{-x}}{\Gamma(r)} \dsP({\rm Gamma(k)} > x)^{d_{n}(v)}\;{\rm d}x \\
& = & A_{1} + A_{2},
\end{eqnarray*}

\noindent where $x_{0}^{\alpha} = a_{n}^{\alpha}$ and $\alpha = \frac{1}{2} \left(\frac{1}{k} + \frac{1}{k+1} \right)$. On the one hand, \Cref{lemma2}, with $q = 1$, implies that 
\begin{eqnarray*}
A_{2} \leq O\left( e^{- \frac{x_{0}^{k}}{2a_{n}k!}} \right) \int_{0}^{\infty} x^{r-1}e^{-x}\;{\rm d}x = O\left( e^{- \frac{x_{0}^{k}}{2a_{n}k!}} \right).
\end{eqnarray*}

\noindent On the other hand, \Cref{lemma2}, with $q = 1$, also implies that 
\begin{eqnarray*}
A_{1} = (1+O(a_{n}^{\frac{1}{2k}})) \int_{0}^{x_{0}} \frac{x^{r-1}e^{-x}}{\Gamma(r)} e^{-\frac{d_{n}(v)x^{k}}{k!}}\;{\rm d} x = (1+O(a_{n}^{\frac{1}{2k}})) \int_{0}^{\infty} \frac{x^{r-1}e^{-x}}{\Gamma(r)} e^{-\frac{d_{n}(v)x^{k}}{k!}}\;{\rm d} x + A_{3},
\end{eqnarray*} 

\noindent where 
\begin{eqnarray*}
A_{3} = \int_{x_{0}}^{\infty} \frac{x^{r-1}e^{-x}}{\Gamma(r)} e^{-\frac{d_{n}(v)x^{k}}{k!}}\;{\rm d} x = O\left( e^{- \frac{x_{0}^{k}}{2a_{n}k!}} \right);
\end{eqnarray*}

\noindent this estimate can be deduced similarly as the one for the integral $A_{2}$. By recalling that $\scK_{r}(\dsT_{n}) = \sum_{v \in \dsT_{n}} I_{r, v}$, we conclude from the previous estimations that
\begin{eqnarray} \label{eq38}
\dsE[\scK_{r}(\dsT_{n}) ] & = & (1+O(a_{n}^{\frac{1}{2k}}))  \sum_{v \in \dsT_{n} \setminus \{\circ \}} \int_{0}^{\infty} \frac{x^{r-1}e^{-x}}{\Gamma(r)} e^{-\frac{d_{n}(v)x^{k}}{k!}}\;{\rm d} x  + o(n a_{n}^{r/k})\\
& = & (1+O(a_{n}^{\frac{1}{2k}})) 2^{-1} \int_{0}^{2(n-1)} \int_{0}^{\infty} \frac{x^{r-1}e^{-x}}{\Gamma(r)} e^{-\frac{\lceil V_{n}(t) \rceil x^{k}}{k!}}\;{\rm d} x {\rm d} t + o(n a_{n}^{r/k}) \nonumber \\
& = & (1+O(a_{n}^{\frac{1}{2k}})) n \int_{0}^{1} \int_{0}^{\infty} \frac{x^{r-1}e^{-x}}{\Gamma(r)} e^{-\frac{ \widehat{V}_{n}(t)  x^{k}}{k!}}\;{\rm d} x {\rm d} t + o(n a_{n}^{r/k}). \nonumber 
\end{eqnarray}

\noindent Finally, our claim follows by making the change of variables $x = a_{n}^{1/k}w$.  
\end{proof}

\begin{lemma} \label{lemma9}
Suppose that $(\dsT_{n})_{n \geq 1}$ is a sequence of ordered (deterministic) rooted trees. Suppose
that there exists a sequence $(a_{n})_{n \geq 1}$ of non-negative real numbers with $\lim_{n
\rightarrow \infty} a_{n} = 0$, $\lim_{n \rightarrow \infty} n a_{n}^{1/k} = \infty$,
and a function $f \in
C([0,1], \dsR_{+})$ such that $\widetilde{V}_{n}$ satisfies the condition (a) in \Cref{lemma3} and
that for $r \in \{1, \dots, k\}$,
\begin{eqnarray*}
\int_{0}^{1} ( a_{n}\widehat{V}_{n}(t) )^{-r/k}\;{\rm d}t  \rightarrow \int_{0}^{1} f(t)^{-r/k}\;{\rm d}t < \infty, \hspace*{3mm} \text{as} \hspace*{2mm} n \rightarrow \infty. 
\end{eqnarray*}

\noindent Then, 
\begin{eqnarray*}
   n^{-1} a_{n}^{-r/k}   \dsE[\scK_{r}(\dsT_{n})] \rightarrow \frac{(k!)^{r/k} \Gamma(r/k)}{k
    \Gamma(r)} \int_{0}^{1} f(t)^{-r/k}\;{\rm d}t, \hspace*{3mm} \text{as} \hspace*{2mm} n
\rightarrow \infty.
\end{eqnarray*}
\end{lemma}

\begin{proof}
Notice that the case $r =1$ has been proved in \Cref{lemma3}. The proof of the general case $r \in \{1, \dots, k\}$ follows by a simple adaptation of the argument used in the proof of \Cref{lemma3} for $q=1$ with the use of \Cref{lemma7}. One only needs to note that
\begin{equation*}
\int_{0}^{1} \int_{0}^{\infty} \frac{x^{r-1}}{\Gamma(r)} e^{-\frac{f(t)  x^{k}}{k!}}\;{\rm d} x {\rm d} t = \frac{(k!)^{\frac{r}{k}}}{k} \frac{\Gamma(\frac{r}{k})}{\Gamma(r)} \int_{0}^{1} f(t)^{-r/k}\;{\rm d}t. 
\qedhere
\end{equation*}
\end{proof}

\section{Proof of \texorpdfstring{\Cref{theo1}}{Theorem 1}}\label{sec:thm1}

Let $\dsT_{n}$ be a Galton-Watson tree conditioned on its number of vertices being $n \in \dsN$ with offspring distribution $\xi$ satisfying (\ref{eq11}). Note that in this case both the $r$-records and the tree are random. Then we study $\scK_{r}(\dsT_{n})$ as random variable conditioned on $\dsT_{n}$. More precisely, we first choose a random tree $\dsT_{n}$. Then we keep it fixed and consider the number of $r$-records. This gives a random variable $\scK_{r}(\dsT_{n})$ with distribution that depends on $\dsT_{n}$. We have the following lemma that corresponds to \cite[Lemma 4.8]{Janson2006}. 

\begin{lemma} \label{lemma8}
Let $\dsT_{n}$ be a Galton-Watson tree conditioned on its number of vertices being $n \in \dsN$ with offspring distribution $\xi$ satisfying (\ref{eq11}). For $r \in \{1, \dots, k\}$. We have that $\dsE[\scK_{r}(\dsT_{n})] = O(n^{1-\frac{r}{2k}})$. 
\end{lemma}

\begin{proof}
By an application of the proof of \Cref{lemma7} with  $a_{n} = n^{-1/2}$ (in particular, the equality (\ref{eq38})), we see that
 \begin{eqnarray}  \label{eq39}
\dsE[\scK_{r}(\dsT_{n})| \dsT_{n} ]  &\leq & (1+O(a_{n}^{\frac{1}{2k}}))  \sum_{v \in \dsT_{n} \setminus \{\circ \}} \int_{0}^{\infty} \frac{x^{r-1}}{\Gamma(r)} e^{-\frac{d_{n}(v)x^{k}}{k!}}\;{\rm d} x  + o(n a_{n}^{r/k}) \nonumber \\
& = &  (1+O(a_{n}^{\frac{1}{2k}}))  \sum_{v \in \dsT_{n} \setminus \{\circ \}} \frac{(k!)^{r/k} \Gamma(r/k)}{k \Gamma(r)}  d_{n}(v)^{-r/k} + o(n a_{n}^{r/k}) \nonumber \\
& = & (1+O(a_{n}^{\frac{1}{2k}})) \frac{(k!)^{r/k} \Gamma(r/k)}{k \Gamma(r)}  \sum_{i=1}^{\infty}   i^{-r/k} w_{i}(\dsT_{n}) + o(n a_{n}^{r/k}),
\end{eqnarray}

\noindent where $w_{i}(\dsT_{n})$ denotes the number of vertices at depth $i \in \dsN$ in $\dsT_{n}$. Notice that
\begin{eqnarray*}
\sum_{i=1}^{\infty}   i^{-r/k} w_{i}(\dsT_{n}) \leq n^{1-\frac{r}{2k}} + \sum_{i=1}^{\lfloor n^{1/2}
\rfloor}   i^{-r/k} w_{i}(\dsT_{n}),
\end{eqnarray*}
\noindent 
by the fact that \(\sum_{i \ge 0} w_{i}(\dsT_{n}) = n\).
Since $\dsE[\xi^{2}] < \infty$ by our assumption \eqref{eq11}, \cite[Theorem 1.13]{Janson2006} implies that for all $n, i \in \dsN$, $\dsE[w_{i}(\dsT_{n})] \leq Ci$ for some constant $C >0$ depending on $\xi$ only. Therefore,
\begin{eqnarray} \label{eq40}
\sum_{i=1}^{\infty}   i^{-r/k} \dsE[w_{i}(\dsT_{n})]
\leq   n^{1-\frac{r}{2k}} + \sum_{i=1}^{\lfloor n^{1/2} \rfloor} \dsE[w_{i}(\dsT_{n})] i^{-\frac{r}{k}}  
= O(n^{1-\frac{r}{2k}}).
\end{eqnarray}

\noindent By taking expectation in (\ref{eq39}), our claim follows by (\ref{eq40}).
\end{proof}

We continue by studying the moments of the number of $1$-records $\scK_{1}(\dsT_{n})$. We denote by $\mu_{n}$ the (random) probability distribution of $\sigma^{-1/k}n^{-1+1/2k}\scK_{1}(\dsT_{n})$ given $\dsT_{n}$. Define the random variables  
\begin{eqnarray*}
m_{q}(\dsT_{n}) \coloneqq \dsE[\scK_{1}(\dsT_{n})^{q} | \dsT_{n}], \hspace*{5mm} q \in \dsZ_{\ge 0}. 
\end{eqnarray*}

\noindent Notice that the moments of $\mu_{n}$ are given by 
$\sigma^{-q/k}n^{-q+q/2k} m_{q}(\dsT_{n})$. We have the following lemma that corresponds to \cite[Lemma 4.9]{Janson2006}.

\begin{lemma}  \label{lemma4}
Let $\dsT_{n}$ be a Galton-Watson tree conditioned on its number of vertices being $n \in \dsN$ with offspring distribution $\xi$ satisfying \eqref{eq11}. Furthermore, suppose that for every fixed $q \in \dsN$ we have that $\dsE[\xi^{q+1}] < \infty$. Then $\dsE[m_{q}(\dsT_{n})] = O(n^{q-\frac{q}{2k}})$. 
\end{lemma}

\begin{proof}
By an application of \Cref{lemma1} with $q \in \dsN$ and $a_{n} = n^{-1/2}$ (in particular, the equality \eqref{eq15} in its proof), we see that
 \begin{eqnarray*} 
m_{q}(\dsT_{n}) \leq (1+O(n^{-\frac{q}{4k}})) q! \sum_{v_{1}, \dots, v_{q} \in \dsT_{n} \setminus \{ \circ \}}  \int_{0}^{\infty} \int_{0}^{x_{1}} \cdots \int_{0}^{x_{q-1}} G_{n}(\vecv_{q}, { \bf x}_{q})\;{\rm d} \cev{\vecx}_{q} + Y_{q},
\end{eqnarray*}

\noindent where $Y_{q} \coloneqq \sum_{p=0}^{q-1} \sum_{l=0}^{p}  \binom{q}{p} \binom{p}{l} (-1)^{p-l} m_{l}(\dsT_{n})$. After a similar computation as in the proof of the inequality (\ref{eq5}), one sees that there exists a constant $C_{k,q} > 0$ such that
 \begin{eqnarray} \label{eq16}
m_{q}(\dsT_{n}) \leq (1+O(n^{-\frac{q}{4k}})) q! C_{k,q} \bar{m}_{1}(\dsT_{n})^{q} + Y_{q},
\end{eqnarray}

\noindent where $\bar{m}_{1}(\dsT_{n})\coloneqq  \sum_{v \in \dsT_{n} \setminus \{ \circ \}} d_{n}(v)^{-1/k}$. Notice that
 \begin{eqnarray*} 
\bar{m}_{1}(\dsT_{n}) = \sum_{i=1}^{\infty} w_{i}(\dsT_{n}) i^{-\frac{1}{k}} \leq  n^{1-\frac{1}{2k}} + \sum_{i=1}^{\lfloor n^{1/2} \rfloor} w_{i}(\dsT_{n}) i^{-\frac{1}{k}},
\end{eqnarray*} 

\noindent where $w_{i}(\dsT_{n})$ denotes the number of vertices at depth $i \in \dsN$ in
$\dsT_{n}$. Since $\dsE[\xi^{q+1}] < \infty$ for $q \in \dsN$, \cite[Theorem
1.13]{Janson2006} implies that for all $n, i \in \dsN$, $\dsE[w_{i}(\dsT_{n})^{q}] \leq
Ci^{q}$ for some constant $C >0$ depending on $q$ and $\xi$ only. Therefore, Minkowski's inequality implies that
\begin{eqnarray} \label{eq17}
\dsE[\bar{m}_{1}(\dsT_{n})^{q}]^{\frac{1}{q}} 
\le n^{1-\frac{1}{2k}} + \sum_{i=1}^{\lfloor n^{1/2} \rfloor} \dsE[w_{i}(\dsT_{n})^{q}]^{\frac{1}{q}} i^{-\frac{1}{k}}  = O(n^{1-\frac{1}{2k}}).
\end{eqnarray}

\noindent By taking expectation in (\ref{eq16}), we deduce from (\ref{eq17}) that
\begin{eqnarray*}
\dsE[m_{q}(\dsT_{n}) ] = \dsE[Y_{q} ]  + O(n^{q-\frac{q}{2k}}),
\end{eqnarray*}

\noindent and our claim follows by induction on $q \in \dsN$. 
\end{proof}

Let $\widetilde{V}_{n}$ and $\widehat{V}_{n}$ be the normalized depth-first search walks associated with the conditioned Galton-Watson tree $\dsT_{n}$. Note that in this case $\widetilde{V}_{n}$ becomes a random function on $C([0,1], \dsR_{+})$. Recall that a remarkable result due to Aldous \cite[Theorem 23 with Remark 2]{Aldous1993} (see also \cite[Theorem 1]{Marck2003}) shows that
\begin{eqnarray} \label{eq12}
n^{-1/2} \widetilde{V}_{n} \inlaw 2 \sigma^{-1} B^{\rm ex}, \hspace*{5mm} \text{as} \hspace*{2mm} n
\rightarrow \infty,
\end{eqnarray}

\noindent in $C([0,1], \dsR_{+})$, with its usual topology, and where $B^{\rm ex} = (B^{\rm ex}(t), t \in [0,1])$ is a standard normalized Brownian excursion. Note that $B^{\rm ex}$ is a random element from $C([0,1], \dsR_{+})$; see for example \cite{Blu1992} or \cite{Revuz1999}. 

\begin{lemma} \label{lemma5}
For $r \in \{1, \dots, k\}$, we have that $\int_{0}^{1} B^{\rm ex}(t)^{-r/k}\;{\rm d}t < \infty$ almost surely. 
\end{lemma}

\begin{proof}
One only needs to show that $\dsE[ \int_{0}^{1} B^{\rm ex}(t)^{-r/k}\;{\rm d}t]<\infty$. This
follows by computing $\dsE[B^{\rm ex}(t)^{-r/k}]$, for every $t \in [0,1]$, from the well-known density
function of $B^{\rm ex}(t)$; see \cite[Chapter II, Equation (1.4)]{Blu1992}. 
\end{proof}

Therefore, \Cref{Theo2} and \Cref{lemma5}  imply that there exists almost surely a
(unique) measure $\nu_{2B^{\rm ex} }$ with moments given by $m_{q}(2B^{\rm ex} )$. The next result provides a generalization of \cite[Theorem 1.10]{Janson2006} and it will be used in the proof of \Cref{theo1}.

\begin{theorem} \label{Theo3}
Let $\dsT_{n}$ be a Galton-Watson tree conditioned on its number of vertices being $n \in \dsN$ with offspring distribution $\xi$ satisfying \eqref{eq11}. Then
\begin{eqnarray} \label{eq13}
\mu_{n} \inlaw \nu_{2B^{\rm ex} }, \hspace*{5mm} \text{as} \hspace*{2mm} n \rightarrow \infty,
\end{eqnarray}

\noindent in the space of probability measures on $\dsR$. Moreover, we have that for every $q \in \dsN$,
\begin{eqnarray} \label{eq14}
\sigma^{-q/k}n^{-q+q/2k} m_{q}(\dsT_{n}) \inlaw  m_{q}(2B^{\rm ex}), \hspace*{5mm} \text{as} \hspace*{2mm} n \rightarrow \infty. 
\end{eqnarray}

\noindent The convergences in \eqref{eq12}, \eqref{eq13} and \eqref{eq14}, for all $q \in \dsN$, hold jointly. 
In particular, if $\dsE[\xi^{p}]<\infty$ for all $p \in \dsN$, then for all $q \in \dsN$ and $l \in \dsN$, 
\begin{eqnarray} \label{eq19}
\sigma^{-lq/k}n^{-lq/k+lq/2k} \dsE[m_{q}(\dsT_{n})^{l}] \rightarrow \dsE[m_{q}(2B^{\rm ex})^{l}], \hspace*{5mm} \text{as} \hspace*{2mm} n \rightarrow \infty.
\end{eqnarray}
\end{theorem}

\begin{proof}
A simple adaptation of the proof of \cite[Lemma 4.7]{Janson2006} easily shows that 
\begin{eqnarray} \label{eq18}
\left(\widetilde{V}_{n}, \int_{0}^{1} \widehat{V}_{n}(t)^{-1/k}\;{\rm d} t \right) \inlaw \left(2
\sigma^{-1} B^{\rm ex}, 2^{-1/k} \sigma^{1/k} \int_{0}^{1}  B^{\rm ex}(t)^{-1/k}\;{\rm d} t \right),
\end{eqnarray}

\noindent in $C([0,1], \mathbb{R}_{+}) \times \dsR$, as $n \rightarrow \infty$. By the Skorohod coupling theorem (see e.g.\ \cite[Theorem
4.30]{Ka2002}), we can assume that the trees $(\dsT_{n})_{n \geq 1}$ are defined on a common
probability space such that the convergence in \eqref{eq18} holds almost surely. Therefore, the
convergences \eqref{eq13} and \eqref{eq14} follow immediately from \Cref{lemma3}. It only remains to
prove \eqref{eq19}. Recall that we assume that $\dsE[\xi^{p}] < \infty$ for every $p \in \dsN$. By
Jensen's inequality, we notice that $m_{q}(\dsT_{n})^{l} \leq m_{lq}(\dsT_{n})$ for $l,q \in \dsN$.
Hence \Cref{lemma4} implies that $\dsE[m_{q}(\dsT_{n})^{l}] = O(n^{lq-\frac{lq}{2k}})$. This shows
that every moment of the right-hand side of \eqref{eq14} stays bounded as $n \rightarrow \infty$
which implies \eqref{eq19}. 
\end{proof}

We are now able to prove \Cref{theo1}. 

\begin{proof}[Proof of \Cref{theo1}]
\Cref{lemma8} establishes that $\mathbb{E} [\scK_{r}(\dsT_{n})] = O(n^{1-\frac{r}{2k}})$ for $r \in \{1, \dots, k \}$. As a consequence, Markov's inequality implies that $n^{-1+ \frac{1}{2k}} \scK_{r}(\dsT_{n}) \rightarrow 0$ in probability, as $n \rightarrow \infty$, for $r \in \{2, \dots, k\}$. Then, by the identity in (\ref{eq32}), it is enough to prove \Cref{theo1} for $\scK_{1}(\dsT_{n})$ instead of $\scK(\dsT_{n})$. By the definition of $\mu_{n}$ and \Cref{Theo3}, for any bounded continuous function $g : \dsR_{+} \rightarrow \dsR_{+}$,
\begin{eqnarray*}
\dsE[g(\sigma^{-1/k}n^{-1+1/2k} \scK_{1}(\dsT_{n})) | \dsT_{n}] = \int g\;{\rm d} \mu_{n} \inlaw \int g\;{\rm d} \nu_{2B^{\rm ex} }, \hspace*{5mm} \text{as} \hspace*{2mm} n \rightarrow \infty.
\end{eqnarray*}

\noindent Taking expectations, the dominated convergence theorem implies that 
$\sigma^{-1/k}n^{-1+1/2k}\scK_{1}(\dsT_{n}) \inlaw Z_{\rm CRT}$, as $n \rightarrow
\infty$, where $Z_{\rm CRT}$ has distribution $\nu(\cdot) = \dsE[\nu_{2B^{\rm ex} }(\cdot)]$.
Suppose that $\dsE[\xi^{p}] < \infty$ for every $p \in \dsN$. \Cref{lemma4} implies
that every moment of $n^{-1+1/2k}\scK_{1}(\dsT_{n})$ stays bounded as $n \rightarrow \infty$
which implies the moment convergence in \Cref{theo1}. It remains
to identify the moments of $Z_{\rm CRT}$ (or equivalently $\nu$). Notice that
\begin{eqnarray*}
\dsE[Z_{\rm CRT}^{q}] = \int x^{q}\;{\rm d} \nu = \dsE \left[ \int x^{q}\;{\rm d} \nu_{2B^{\rm ex} }\right] = \dsE[m_{q}(2B^{\rm ex})], \hspace*{5mm} \text{for} \hspace*{2mm} q \in \dsN.
\end{eqnarray*}

\noindent For $q \in \dsN$, let $U_{1}, \dots, U_{q}$ be independent random variables with the
uniform distribution on $[0,1]$. Let $Y_{1}, \dots, Y_{q}$ be the first $q$ points in a Poisson
process on $(0, \infty)$ with intensity $x\,{\rm d} x$, i.e., $Y_{1}, \dots, Y_{q}$ have joint
density function $y_{1} \cdots y_{q} e^{-y_{q}^{2}/2}$  on $0 < y_{1} < \cdots < y_{q} < \infty$.
It is well-known that $L_{2B^{\rm ex}}(U_{1}, \dots, U_{q}) \stackrel{d}{=} Y_{q}$, see, e.g.,
\cite[Proof of Lemma 5.1]{Janson2006}. Thus by recalling the definition of the
function $H_{2B^{\rm ex}, q}$ in \eqref{eq3}, we see that 
\begin{eqnarray} \label{eq20}
\dsE[m_{q}(2B^{\rm ex})] = q ! \dsE[H_{2B^{\rm ex}, q}(\vecU_{q})] = q! \int_{0}^{y_{1}} \cdots \int_{0}^{y_{q-1}} \int_{0}^{\infty} y_{1} \cdots y_{q} e^{-y_{q}^{2}/2 } \widetilde{F}_{q}(\vecy_{q})\;{\rm d} \vecy_{q},
\end{eqnarray}

\noindent where $\vecU_{q} = (U_{1}, \dots, U_{q})$, $\vecy_{q} = (y_{1}, \dots, y_{q}) \in \dsR_{+}^{q}$ and
\begin{eqnarray*}
\widetilde{F}_{q}(\vecy_{q}) \coloneqq  \int_{0}^{\infty} \int_{0}^{x_{1}} \dots \int_{0}^{x_{q-1}} \exp \left( - \frac{y_{1}x_{1}^{k} + (y_{2}-y_{1})x_{1}^{k} +\dots + (y_{q}-y_{q-1})x_{q}^{k} }{k !} \right)\;{\rm d} \cev{\vecx}_{q}.
\end{eqnarray*}

\noindent Finally, the expression for the moments in \Cref{theo1} follows by first changing the order of integration in \eqref{eq20} and then by making the change of variables $w_{i} = y_{i}-y_{i-1}$ for $2 \leq  i \leq q$. 
\end{proof}

Following the idea of the proof of \Cref{theo1}, we obtain the following convergence of the first
moment of the number of $r$-records $\scK_{r}(\dsT_{n})$. This provides a proof of \cite[Lemma
4.10]{Cai2018}. 

\begin{lemma}
Let $\dsT_{n}$ be a Galton-Watson tree conditioned on its number of vertices being $n \in \dsN$ with offspring distribution $\xi$ satisfying (\ref{eq11}). For $r \in \{1, \dots k\}$, we have that
\begin{eqnarray*}
n^{-1+\frac{r}{2k}} \dsE[\scK_{r}(\dsT_{n})] \rightarrow \frac{(k!)^{\frac{r}{k}}}{k} \frac{\Gamma(\frac{r}{k}) \Gamma \left(1- \frac{r}{2k} \right) }{\Gamma(r)} \left( \frac{\sigma}{\sqrt{2}}\right)^{\frac{r}{k}}  , \hspace*{3mm} \text{as} \hspace*{2mm} n \rightarrow \infty.
\end{eqnarray*}
\end{lemma} 

\begin{proof}
The proof follows by a simple adaptation of the argument used in the proof of \Cref{theo1} by using \Cref{lemma9} (with $a_{n}=n^{-1/2}$), \Cref{lemma8} and \Cref{lemma5}. One only needs to note that 
\begin{eqnarray*}
\dsE \left[ \int_{0}^{1} B^{\rm ex}(t)^{-r/k}\;{\rm d}t \right] = 2^{-\frac{r}{2k}} \Gamma \left(1- \frac{r}{2k} \right),
\end{eqnarray*}

\noindent which follows from the well-known density function of $B^{\rm ex}(t)$; see \cite[Chapter II, Equation (1.4)]{Blu1992}. 
\end{proof}

\section{Further applications} \label{sec4}

In this section, we show that the results obtained in \Cref{sec1} can be used and extended to
study the $k$-cut model in other families of trees. In this section, let $\dsT_{n}$ be a rooted tree
(maybe random and not necessarily ordered) with $n \in \dsN$ vertices and root $\circ$. 

\subsection{Paths}

\begin{lemma}\label{lem:path}
Let $\dsT_{n}$ be a path with $n$ vertices labelled $1, \dots, n$ from the root to the leaf. For $k \in \{2, 3, \dots \}$, we have that
$n^{-1+1/k} \scK(\dsT_{n})  \inlaw Z_{\rm path}$, as $n \rightarrow \infty$,
where $Z_{\rm path}$ is a non-degenerate random variable whose law is determined entirely by its 
moments: $\dsE[Z_{\rm path}^{q}]  = m_{q}(f)$ for $q \in \dsZ_{\ge 0}$,
where
\[
f(t) = \left\{ \begin{array}{lcl}
              2t, & t \in [0,1/2], \\
              2-2t,  & t \in (1/2,1]. \\
              \end{array}
    \right.
\]
\end{lemma}
\begin{proof}
    By \cite[Theorem 1.1]{Cai2018}, we know that $\dsE[\scK_{r}(\dsT_{n})] =
O(n^{1-\frac{r}{k}})$, for $r \in \{1, \dots, k-1\}$, and $\dsE[\scK_{k}(\dsT_{n})] = O(\ln
n)$. Then Markov's inequality implies that $n^{-1+1/k} \scK_{r}(\dsT_{n}) \rightarrow 0$
in probability, as $n \rightarrow \infty$, for $r \in \{2, \dots, k \}$. Thus, by the identity
(\ref{eq32}), it is enough to prove our result for $\scK_{1}(\dsT_{n})$ instead of
$\scK(\dsT_{n})$. 
Note that the normalized depth-first search walks $\widetilde{V}_{n}$ and $\widehat{V}_{n}$ of $\dsT_{n}$, defined in (\ref{eq35}), are given by 
\( n^{-1}\widetilde{V}_{n}(t) = f(t).  \)
\noindent and that $n^{-1}\widehat{V}_{n}(t)= n^{-1}\lceil \widetilde{V}_{n}(t) \rceil$ for $t \in
[0,1]$. It should be clear that the conditions of \Cref{lemma3} are fulfilled with $a_{n} = n^{-1}$. Therefore, our result follows from a simple application of 
\Cref{lemma3}. 
\end{proof}

\begin{remark}
The convergence in distribution and moments of the $k$-cut number of a path to \(Z_{\rm path}\) has
been proved in \cite[Theorem 1.5]{Cai2018} with a very different method. The contribution of
\Cref{lem:path} is the formula  for computing the \(q\)-th moment of the limiting variable \(Z_{\rm
path}\) for all \(q \in \dsZ_{\ge 0}\).
\end{remark}

\subsection{General trees}

The next result establishes a limit in distribution for the number of $1$-records
$\scK_{1}(\dsT_{n})$ of a general (random) rooted tree in the same spirit as in
\Cref{lemma3}. For $q \in \dsN$, let $u_{1}, \dots, u_{q}$ be a sequence of independent uniformly
chosen vertices on $\dsT_{n}$. Recall that $L_{n}(u_{1}, \dots, u_{q})$ denotes the number of edges
in the subtree of $\dsT_{n}$ spanned by $u_{1}, \dots, u_{q}$ and its root $\circ$ (i.e., the minimal number of edges that are needed to connect $u_{1}, \dots, u_{q}$ and $\circ$). In particular,
$L_{n}(u_{1}) = d_{n}(u_{1})$ is the depth of the vertex $u_{1}$ in $\dsT_{n}$. In the sequel, we
will often use the notation $A_{n} = O_{p}(B_{n})$, where $(A_{n})_{n \geq 1}$ and $(B_{n})_{n \geq
    1}$ are two sequences of non-negative real random variables such that $B_{n} >0$, to indicate
that $\lim_{\delta \rightarrow \infty} \limsup_{n \rightarrow \infty} \dsP(A_{n} > \delta B_{n})
=0$. 

\begin{theorem} \label{Theo4}
Let $(\dsT_{n})_{n \geq 1}$ be a sequence of rooted trees. Suppose that there exists a sequence $(a_{n})_{n \geq 1}$ of non-negative real numbers with $\lim_{n \rightarrow \infty} a_{n} = 0$, $\lim_{n \rightarrow \infty} na_{n}^{1/k} = \infty$ and such that
\begin{itemize}
\item[(a)] $ \max_{v \in \dsT_{n}} L_{n}(v) = O_{\rm p}(a_{n}^{-1}).$ 

\item[(b)] For every $q \in \dsN$, $ \displaystyle a_{n} (L_{n}(u_{1}), \dots,L_{n}(u_{1}, \dots, u_{q}))  \inlaw (\zeta_{1}, \dots, \zeta_{1} + \cdots + \zeta_{q}), \hspace*{1mm} \text{as} \hspace*{1mm} n \rightarrow \infty$, where $\zeta_{1}, \zeta_{2} \dots$ is a sequence of i.i.d.\ random variables in $\dsR_{+}$ with no atom at $0$. 

\item[(c)] For every $q \in \dsN$, $ \displaystyle \dsE[(a_{n}L_{n}(u_{1}) \cdots a_{n}L_{n}(u_{q}))^{-1/k} \mathds{1}_{\{ u_{1}, \dots, u_{q} \in \dsT_{n}  \setminus \{ \circ \} \}} ] \rightarrow \dsE[\zeta_{1}^{-1/k}]^{q} < \infty, \hspace*{1mm} \text{as} \hspace*{1mm} n \rightarrow \infty.$
\end{itemize}

\noindent Then $n^{-1}     a_{n}^{-1/k} \scK_{1}(\dsT_{n})  \inlaw Z_{\zeta}$, as $n
\rightarrow \infty$, where $Z_{\zeta}$ is a random variable whose law is determined entirely by its moments: $\dsE[Z_{\zeta}^{0}]=1$, and for $q \in \dsN$,
\begin{eqnarray*}
\dsE[Z_{\zeta}^{q}] = q! \int_{0}^{\infty} \int_{0}^{x_{1}} \cdots \int_{0}^{x_{q-1}}
\dsE\left[\exp\left({- \frac{\zeta_{1}x^{k}_{1} + \cdots + \zeta_{q}x_{q}^{k}}{k!}}  \right)
\right]\;{\rm d} \cev{\bf x}_{q}
\end{eqnarray*}  
\end{theorem}

\begin{proof}
By the assumption (a) and \Cref{lemma1} (in particular, the identity \eqref{eq15}), we see that 
\begin{eqnarray*} 
\dsE[\scK_{1}(\dsT_{n})^{q} | \dsT_{n}] = (1 + O( a_{n}^{\frac{q}{2k}} ) ) q! \sum_{v_{1}, \dots, v_{q} \in \dsT_{n} \setminus \{ \circ \}} \widehat{H}_{n,q}(\vecv_{q}) + Y_{q},
\end{eqnarray*}

\noindent where $\vecv_{q} = (v_{1}, \dots, v_{q}) \in \dsT_{n}^{q}$,  
$Y_{q} \coloneqq \sum_{p=0}^{q-1} \sum_{l=0}^{p}  \binom{q}{p} \binom{p}{l} (-1)^{p-l} \dsE[\scK_{1}(\dsT_{n})^{l} | \dsT_{n}]$ and
\begin{eqnarray*}
\widehat{H}_{n,q}(\vecv_{q}) \coloneqq  \int_{0}^{\infty} \int_{0}^{x_{1}} \cdots \int_{0}^{x_{q-1}} G_{n}(\vecv_{q}, { \bf x}_{q})  e^{- \sum_{i=1}^{q}x_{i}}\;{\rm d} \cev{\vecx}_{q},
\end{eqnarray*}

\noindent with $G_{n}$ defined in (\ref{eq29}). Then we see that
\begin{eqnarray*} 
n^{-q}\dsE[\scK_{1}(\dsT_{n})^{q}] =   (1 + O( a_{n}^{\frac{q}{2k}} )) q!    \dsE[\widehat{H}_{n,q}(\vecu_{q}) \mathds{1}_{\{ \vecu_{q} \in (\dsT_{n}  \setminus \{ \circ \})^{q}  \}}] + n^{-q}\dsE[Y_{q}] ,
\end{eqnarray*}

\noindent where $\vecu_{q} = (u_{1}, \dots, u_{q})$. Suppose that we have proven that 
\begin{eqnarray} \label{eq31}
a_{n}^{-q/k}\dsE[\widehat{H}_{n,q}(\vecu_{q}) \mathds{1}_{\{ \vecu_{q} \in (\dsT_{n}  \setminus \{
    \circ \})^{q}  \}}] \rightarrow  \int_{0}^{\infty} \int_{0}^{x_{1}} \cdots \int_{0}^{x_{q-1}}
\dsE\left[\exp\left({- \frac{\zeta_{1}x^{k}_{1} + \cdots + \zeta_{q}x_{q}^{k}}{k!}}  \right) \right]\;{\rm d}
\cev{\bf x}_{q}
.
\end{eqnarray}

\noindent as $n \rightarrow \infty$. Then the result follows by induction on  $q \in \dsN$ together with the previous convergence.

We henceforth prove the claim in \eqref{eq31}. From the result in \eqref{eq9}, it is enough to check the following:
\begin{itemize}
\item[(i)] The sequence $(a_{n}^{-q/k}\widehat{H}_{n,q}(\vecu_{q}) \mathds{1}_{\{ \vecu_{q} \in (\dsT_{n}  \setminus \{ \circ \})^{q} \} })_{n \geq 1}$ is uniformly integrable.

\item[(ii)] $ \displaystyle a_{n}^{-q/k}\widehat{H}_{n,q}(\vecu_{q}) \mathds{1}_{\{ \vecu_{q} \in
        (\dsT_{n}  \setminus \{ \circ \})^{q} \} } \inlaw \int_{0}^{\infty}
    \int_{0}^{x_{1}} \cdots \int_{0}^{x_{q-1}}
    \exp\left({- \frac{\zeta_{1}x^{k}_{1} + \cdots +
            \zeta_{q}x_{q}^{k}}{k!}}  \right)\;{\rm d} \cev{\bf x}_{q}$, as $n
    \rightarrow \infty$. 
\end{itemize}

We start by showing (i). Since $\exp(-(x_{1}+\cdots+ x_{q})) \leq 1$ for $x_{1}, \dots, x_{q} \in \dsR_{+}$, we have that 
\begin{eqnarray*}
\widehat{H}_{n,q}(\vecu_{q}) \leq  \int_{0}^{\infty} \int_{0}^{x_{1}} \cdots \int_{0}^{x_{q-1}} G_{n}(\vecu_{q}, { \bf x}_{q})\;{\rm d} \cev{\vecx}_{q}
.
\end{eqnarray*}

\noindent Hence after a similar computation as in the proof of the inequality \eqref{eq5}, one obtains that there exists a constant $C_{k,q} > 0$ such that 
\begin{eqnarray*}
a_{n}^{-q/k} \widehat{H}_{n,q}(\vecu_{q}) \mathds{1}_{\{ \vecu_{q} \in (\dsT_{n}  \setminus \{ \circ \})^{q} \} }  \leq   C_{k,q} a_{n}^{-q/k}  (L_{n}(u_{1}) \cdots L_{n}(u_{q}))^{-1/k} \mathds{1}_{\{ \vecu_{q} \in (\dsT_{n}  \setminus \{ \circ \})^{q} \}}.
\end{eqnarray*}

\noindent Notice that our hypotheses (b) and (c) together with the result in \eqref{eq9} show that
the sequence 
\begin{eqnarray*}
a_{n}^{-q/k} ( (L_{n}(u_{1}) \cdots L_{n}(u_{q}))^{-1/k} \mathds{1}_{\{ \vecu_{q} \in (\dsT_{n}  \setminus \{ \circ \})^{q} \}})_{n \geq 1}
\end{eqnarray*}

\noindent is uniformly integrable. Hence (i) follows from \cite[Theorem 5.4.5]{Gut2013}.  

Finally, we verify (ii). By making the change of variables $x_{i} = a_{n}^{1/k} w_{i}$, for $1 \leq i \leq q$, we see that
\begin{eqnarray*}
\widehat{H}_{n,q}(\vecu_{q}) = a_{n}^{q/k} \int_{0}^{\infty} \int_{0}^{w_{1}} \cdots \int_{0}^{w_{q-1}} \bar{G}_{n}(\vecu_{q}, { \bf w}_{q})  e^{- a_{n}^{1/k}\sum_{i=1}^{q}w_{i}}\;{\rm d} \cev{\vecw}_{q},
\end{eqnarray*}

\noindent where $\vecw_{q} = (w_{1}, \dots, w_{q}) \in \dsR_{+}^{q}$, $\cev{\vecw}_{q} = (w_{q}, \dots, w_{1})$, and
\begin{eqnarray*} 
\bar{G}_{n}(\vecu_{q}, \vecw_{q}) \coloneqq \exp \left( - \frac{a_{n}D_{n}(u_{1})w_{1}^{k} + \dots + a_{n}D_{n}(u_{1}, \dots, u_{q})w_{q}^{k} }{k !} \right), 
\end{eqnarray*}

\noindent with $D_{n}(u_{1}) \coloneqq L_{n}(u_{1})$ and $D_{n}(u_{1}, \dots, u_{q}) \coloneqq L_{n}(u_{1}, \dots, u_{q}) - L_{n}(u_{1}, \dots, u_{q-1})$ for $q \geq 2$. Notice  that $\mathds{1}_{\{ \vecu_{q} \in (\dsT_{n}  \setminus \{ \circ \})^{q} \} } \inlaw 1$, as $n \rightarrow \infty$. Thus, condition (b) implies that 
\begin{eqnarray*}
\bar{G}_{n}(\vecu_{q}, \vecw_{q}) \mathds{1}_{\{ \vecu_{q} \in (\dsT_{n}  \setminus \{ \circ \})^{q}
    \} } \inlaw  \exp\left({- \frac{\zeta_{1}w^{k}_{1} + \cdots + \zeta_{q}w_{q}^{k}}{k!}}  \right),  \hspace*{3mm} \text{as} \hspace*{2mm} n \rightarrow \infty
\end{eqnarray*}

\noindent By the Skorohod coupling theorem (see e.g.\ \cite[Theorem 4.30]{Ka2002}), we can assume that the
previous convergence holds almost surely together with the convergence in condition (b). Notice that
for $\varepsilon \in (0,1)$ there exists $N \in \dsN$ such that 
$$\bar{G}_{n}(\vecv_{q}, \vecw_{q}) \mathds{1}_{\{ \vecv_{q} \in (\dsT_{n}  \setminus \{ \circ \})^{q} \} }  \leq 
\exp\left({-(1- \varepsilon) \zeta_{1}w_{1}^{k}/k!}  \right), \hspace*{3mm} \text{for} \hspace*{2mm} n \geq N. 
$$ 
By condition (c), notice also that the function on the
right-hand side is integrable on $\{ \vecw_{q} \in \dsR_{+}^{q}: 0 \leq w_{q} \leq \cdots \leq w_{1}
< \infty \}$. Therefore, it should be clear now that (ii) follows by the dominated convergence
theorem. This concludes our proof. 
\end{proof}

The next result establishes an estimate for the mean number of $r$-records $\scK_{r}(\dsT_{n})$ of a
general (random) rooted tree in the same spirit as in \Cref{lemma9}. Furthermore, it shows that
$\scK_{r}(\dsT_{n})$ is of smaller order than $\scK_{1}(\dsT_{n})$ and hence it will not contribute
(in the limit) to the distribution of the $k$-cut number $\scK(\dsT_{n})$. We believe as well that
our methods can be used to estimate higher moments and to obtain an analogue result to \Cref{Theo4} for $\scK_{r}(\dsT_{n})$. We have not attempted to do it and the estimation of the mean is enough for our purpose. 

\begin{lemma} \label{lemma10}
Let $(\dsT_{n})_{n \geq 1}$ be a sequence of rooted trees. Suppose that there exists a sequence $(a_{n})_{n \geq 1}$ of non-negative real numbers with $\lim_{n \rightarrow \infty} a_{n} = 0$, $\lim_{n \rightarrow \infty} na_{n} = \infty$ and such that
\begin{itemize}
\item[(a)] $ \max_{v \in \dsT_{n}} L_{n}(v) = O_{\rm p}(a_{n}^{-1}).$ 

\item[(b)] $ \displaystyle a_{n} L_{n}(u_{1})  \inlaw \zeta_{1}, \hspace*{1mm} \text{as} \hspace*{1mm} n \rightarrow \infty$, where $\zeta_{1}$ is a random variable in $\dsR_{+}$ with no atom at $0$. 

\item[(c)] For every $r \in \{1, \dots k\}$, $ \displaystyle \dsE[(a_{n}L_{n}(u_{1}))^{-r/k} \mathds{1}_{\{ u_{1} \in \dsT_{n}  \setminus \{ \circ \} \}} ] \rightarrow \dsE[\zeta_{1}^{-r/k}] < \infty, \hspace*{1mm} \text{as} \hspace*{1mm} n \rightarrow \infty.$
\end{itemize}

\noindent Then, for $r \in \{1, \dots k\}$,
\begin{eqnarray} \label{eq42}
   n^{-1} a_{n}^{-r/k}   \dsE[\scK_{r}(\dsT_{n})] \rightarrow \frac{(k!)^{r/k} \Gamma(r/k)}{k
    \Gamma(r)}\dsE[\zeta_{1}^{-r/k}], \hspace*{3mm} \text{as} \hspace*{2mm} n
\rightarrow \infty.
\end{eqnarray}
\end{lemma}

\begin{proof}
By the assumption (a) and \Cref{lemma7} (in particular, the identity \eqref{eq38}), we see that
\begin{eqnarray*} 
\dsE[\scK_{r}(\dsT_{n}) | \dsT_{n}] =   (1+O(a_{n}^{\frac{1}{2k}}))  \sum_{v \in \dsT_{n} \setminus \{\circ \}} \int_{0}^{\infty} \frac{x^{r-1}e^{-x}}{\Gamma(r)} e^{-\frac{L_{n}(v)x^{k}}{k!}}\;{\rm d} x  + o(n a_{n}^{r/k}). \\
\end{eqnarray*}

\noindent Hence
\begin{eqnarray*}
n^{-1}\dsE[\scK_{r}(\dsT_{n})]  = (1+O(a_{n}^{\frac{1}{2k}})) \int_{0}^{\infty} \frac{x^{r-1}e^{-x}}{\Gamma(r)} \mathbb{E}\left[e^{-\frac{L_{n}(u_{1})x^{k}}{k!}} \mathds{1}_{\{ u_{1} \in \dsT_{n}  \setminus \{ \circ \} \}} \right] \;{\rm d} x  + o(a_{n}^{r/k}).
\end{eqnarray*}

\noindent Therefore, our result follows by proving that 
\begin{eqnarray*}
a_{n}^{-r/k} \int_{0}^{\infty} \frac{x^{r-1}e^{-x}}{\Gamma(r)} \mathbb{E}\left[e^{-\frac{L_{n}(u_{1})x^{k}}{k!}} \mathds{1}_{\{ u_{1} \in \dsT_{n}  \setminus \{ \circ \} \}} \right] \;{\rm d} x  \rightarrow  \int_{0}^{\infty} \frac{x^{r-1}}{\Gamma(r)} \mathbb{E} \left[ e^{-\frac{\zeta_{1}  x^{k}}{k!}}\; \right] {\rm d} x, \hspace*{3mm} \text{as} \hspace*{2mm} n \rightarrow \infty,
\end{eqnarray*}

\noindent where the last integral is equal to the right-hand side of (\ref{eq42}). Note that the case $r =1$ has been proved in \Cref{Theo4}. The proof of the general case $r \in \{1, \dots, k\}$ follows by a simple adaptation of the argument used in the proof of \Cref{Theo4} for $q=1$ and details are left to the reader.
\end{proof}

The next lemma provides a useful way to verify condition (c) in \Cref{Theo4}.

\begin{lemma} \label{lemma6}
Let $\dsT_{n}$ be a rooted tree. Suppose that there exists a sequence $(a_{n})_{n \geq 1}$ of
non-negative real numbers with $\lim_{n \rightarrow \infty} a_{n} = 0$, $\lim_{n \rightarrow \infty}
na_{n}^{1/k} = \infty$ and such that for every $q \in \dsN$, 
\begin{eqnarray*}
a_{n}(L_{n}(u_{1}), \cdots, L_{n}(u_{q}))  \inlaw (\zeta_{1}, \dots, \zeta_{q}), \hspace*{3mm} \text{as} \hspace*{2mm} n \rightarrow \infty,
\end{eqnarray*}

\noindent where $\zeta_{1}, \zeta_{2} \dots$ is a sequence of i.i.d.\ random variables in $\dsR_{+}$ with no atom at $0$ such that $\dsE[\zeta_{1}^{-1/k}] < \infty$. Furthermore, assume that for every $q \in \dsN$ there exists $\delta >0$ such that for all $\varepsilon \in (0, \delta)$
\begin{eqnarray} \label{eq34}
\dsE[W_{i}(\dsT_{n})] = o( n a_{n}^{q/k+1}), \hspace*{3mm} \text{uniformly on} \hspace*{3mm} 0 \leq i \leq \varepsilon^{1/q} a_{n}^{-1},
\end{eqnarray}

\noindent where $W_{i}(\dsT_{n})$ denotes the number of vertices a depth $i \in \dsZ_{\ge 0}$ in $\dsT_{n}$. Then the condition (c) in \Cref{Theo4} is satisfied
\end{lemma}

\begin{proof}
For simplicity, we introduce the notation $X_{n,q} \coloneqq (a_{n}L_{n}(u_{1}) \cdots a_{n}L_{n}(u_{q}))^{-1/k}\mathds{1}_{\{ \vecu_{q} \in (\dsT_{n}  \setminus \{ \circ \})^{q} \} }$ and $X_{q} \coloneqq (\zeta_{1} \cdots \zeta_{q})^{-1/k}$, for $n, q \in \dsN$. Consider $\delta >0$ such that for $\varepsilon \in (0,\delta)$ the property in \eqref{eq34} is satisfied. Define the function $\phi_{\varepsilon}: \dsR_{+} \rightarrow \dsR_{+}$ given by $\phi_{\varepsilon} = 0$ on $[0,\varepsilon]$, $\phi_{\varepsilon} = 1$ on $[2 \varepsilon, \infty)$, and $\phi_{\varepsilon}$ linear on $[\varepsilon, 2 \varepsilon]$.
Since $\mathds{1}_{\{ \vecu_{q} \in (\dsT_{n}  \setminus \{ \circ \})^{q} \} } \inlaw 1$ we observe that
\begin{eqnarray*}
 \dsE[X_{n,q} \phi_{\varepsilon}(X_{n,q}^{-k})] \rightarrow \dsE[X_{q} \phi_{\varepsilon}(X^{-k}_{q})], \hspace*{3mm} \text{as} \hspace*{2mm} n \rightarrow \infty.
\end{eqnarray*} 

\noindent Further, we note that $\phi_{\varepsilon}(X^{-k}_{q}) \rightarrow 1$, almost surely, as $\varepsilon \rightarrow 0$. In order to show that condition (c) in \Cref{Theo4} is fulfilled, it is enough to check that 
\begin{eqnarray} \label{eq37}
\lim_{\varepsilon \rightarrow 0} \lim_{n \rightarrow \infty} \dsE[(X_{n,q} -X_{n,q} \phi_{\varepsilon}(X_{n,q}^{-k}))] = 0. 
\end{eqnarray}

\noindent Notice that
\begin{eqnarray*}
\dsE[(X_{n,q} -X_{n,q} \phi_{\varepsilon}(X_{n,q}^{-k}))] \leq \dsE \left[X_{n,q} \mathds{1}_{ \{ X_{n,q}^{-k} \leq  \varepsilon \}} \right].
\end{eqnarray*}

\noindent Since $\{ X_{n,q}^{-k} \leq  \varepsilon \} \subseteq \{ 1 \leq L_{n}(u_{1}) \leq \varepsilon^{1/q} a_{n}^{-1} \} \cap \cdots \cap \{ 1 \leq L_{n}(u_{q}) \leq \varepsilon^{1/q} a_{n}^{-1} \}$, it is not difficult to see that
\begin{eqnarray*}
\dsE[(X_{n,q} -X_{n,q} \phi_{\varepsilon}(X_{n,q}^{-k}))] & \leq & \dsE\left[ \left( \frac{1}{n}\sum_{v \in \dsT_{n} \setminus \{ \circ\}} (a_{n}L_{n}(v))^{-1/k} \mathds{1}_{ \{ L_{n}(v) \leq  \varepsilon^{1/q} a_{n}^{-1} \}} \right)^{q}  \right] \\
& \leq & \frac{1}{n}\dsE\left[ \sum_{v \in \dsT_{n} \setminus \{ \circ\}} (a_{n}L_{n}(v))^{-q/k} \mathds{1}_{ \{ L_{n}(v) \leq  \varepsilon^{1/q} a_{n}^{-1} \}}  \right],
\end{eqnarray*}

\noindent where we have used Jensen's inequality to obtain the second inequality. Finally, by our choice of $\varepsilon$ (recall assumption \eqref{eq34}), we observe that 
\begin{eqnarray*}
\dsE[(X_{n,q} -X_{n,q} \phi_{\varepsilon}(X_{n,q}^{-k}))] & \leq & n^{-1} a_{n}^{-q/k}
\sum_{i=1}^{\lfloor \varepsilon^{1/q} a_{n}^{-1} \rfloor} i^{-q/k} \dsE[W_{i}(\dsT_{n})] = o(1).
\end{eqnarray*}

\noindent This clearly implies \eqref{eq37} and concludes our proof.
\end{proof}

Similarly, we also provide a useful way to verify condition (c) in \Cref{lemma10}.

\begin{lemma} \label{lemma11}
Let $\dsT_{n}$ be a rooted tree. Suppose that there exists a sequence $(a_{n})_{n \geq 1}$ of
non-negative real numbers with $\lim_{n \rightarrow \infty} a_{n} = 0$, $\lim_{n \rightarrow \infty}
na_{n} = \infty$ and such that the condition (b) in \Cref{lemma10} holds with a random variable
$\zeta_{1}$ satisfying $\dsE[\zeta_{1}^{-r/k}] < \infty$ for every $r \in \{1, \dots, k\}$.
Furthermore, assume that for every $r \in \{1, \dots, k\}$ there exists $\delta >0$ such that for all
$\varepsilon \in (0, \delta)$
\begin{eqnarray} \label{eq34a}
\dsE[W_{i}(\dsT_{n})] = o( n a_{n}^{r/k+1}), \hspace*{3mm} \text{uniformly on} \hspace*{3mm} 0 \leq i \leq \varepsilon a_{n}^{-1},
\end{eqnarray}

\noindent where $W_{i}(\dsT_{n})$ denotes the number of vertices at depth $i \in \dsZ_{\ge 0}$ in $\dsT_{n}$. Then the condition (c) in \Cref{lemma10} is fulfilled.
\end{lemma}

\begin{proof}
It should be clear that this can be shown along the lines of the proof of \Cref{lemma6}, and therefore, we omit its proof. 
\end{proof}
 
\subsection{Trees of logarithmic height} \label{sect43}
 
Natural examples of trees that fulfil the conditions of \Cref{Theo4} are the class of random trees
with logarithmic height, i.e., trees $\dsT_{n}$ such that $\max_{v \in \dsT_{n}} d_{n}(v)=O_{\rm p}(\ln n)$. For instance, random split trees, uniform random recursive trees, scale-free random trees and mixtures of complete regular trees.  

\subsubsection{Complete binary trees} 

Let $\dsT_{n}^{\rm bi}$ be a complete binary tree with $n \in
\dsN$ vertices, i.e., its height is $\lfloor \ln n \rfloor$. Recall that $\dsT_{n}^{\rm bi}$ has
$2^{i}$ vertices at height $i \in \{0, 1, \dots, \lfloor \ln n \rfloor\}$, and $n-2^{\lfloor \ln n
    \rfloor}+1$ vertices of height $\lfloor \ln n \rfloor$, moreover, the vertices of height
$\lfloor \ln n \rfloor$ have leftmost positions among the $2^{\lfloor \ln n \rfloor}$ possible ones;
see, e.g., \cite[Page 401]{Knuth1997}.  We use the notation $\lg_{2}n = (\ln n) /(\ln 2)$ for the logarithm with base $2$ of $n \in \mathbb{N}$. It should be clear that condition (a) in \Cref{Theo4}
is satisfied with $a_{n} = (\lg_{2} n)^{-1}$. Furthermore, one readily checks that $(\lg_{2} n)^{-1}(L_{n}(u_{1}), L_{n}(u_{1}, u_{2})) \inlaw (1, 2)$, as $n
\rightarrow \infty$. By a simple application of \cite[Corollary 1]{bertoin2013}, this implies that
condition (b) in \Cref{Theo4} is satisfied with
$\zeta_{1} \equiv 1$. Notice that each vertex in $\dsT_{n}^{\rm bi}$ has at most $2$ children. Then
it should be clear that condition (c) of \Cref{Theo4} follows from \Cref{lemma6} since
$\dsE[W_{i}(\dsT_{n}^{\rm bi})] \leq 2^{i}$ for $i \in \dsZ_{\ge 0}$. Therefore, \Cref{Theo4}
implies that $ n^{-1}  (\lg_{2}n)^{1/k} \scK_{1}(\dsT_{n}^{\rm bi})  \inlaw
Z_{1}$, as $n \rightarrow \infty$, where $Z_{1}$ is the random variable whose law is
determined entirely by its moments: $\dsE[Z_{1}^{0}] = 1$, and for $q \in \dsN$,
\begin{eqnarray} \label{eq33}
\dsE[Z_{1}^{q}] = q! \int_{0}^{\infty} \int_{0}^{x_{1}} \cdots \int_{0}^{x_{q-1}} 
\exp\left( 
    {- \frac{x^{k}_{1} + \cdots + x_{q}^{k}}{k!}}
\right)\;{\rm d} \cev{\bf x}_{q}. 
\end{eqnarray} 

\noindent It should be clear that \Cref{lemma10} and \Cref{lemma11} imply that $\dsE[\scK_{r}(\dsT_{n}^{\rm bi}) ] = O(n (\lg_{2} n)^{-r/k})$ for $r \in \{1, \dots, k\}$. Therefore, by the identity \eqref{eq32} and the Markov's inequality, $n^{-1} (\lg_{2} n)^{1/k} \scK(\dsT_{n}^{\rm bi})  \inlaw Z_{1}$, as $n \rightarrow \infty$.
However, it follows from the next lemma that \(Z_{1} \equiv (k!)^{\frac{1}{k}} \Gamma
    \left(1+1/k\right)\). Therefore, we actually have 
\begin{equation}\label{eq:bin}
    n^{-1} (\lg_{2} n)^{1/k} \scK(\dsT_{n}^{\rm bi}) \inlaw 
    Z_{1} \equiv
    (k!)^{{1}/{k}} \Gamma \left(1 + 1/k\right).
\end{equation}

\begin{remark}
    As Theorem 1.1 of \cite{Cai20182} shows, \(\scK(\dsT^{\rm bi})\), after proper shifting and
    rescaling, also converges to a non-degenerate limit distribution with an infinite mean. Thus it
    is not possible to derive the result in \cite{Cai20182} with the method of moments which we use
    to derive \Cref{theo1} for conditioned Galton-Watson trees.  The same is true for split trees,
    random recursive trees and
    scale-free trees.
\end{remark}

\begin{lemma}
    For \(q \in \dsN\), we have that
    \begin{equation*}\label{eq:bin:int}
        q! \int_{0}^{\infty} \int_{0}^{x_{1}} \cdots \int_{0}^{x_{q-1}} 
        \exp\left( 
            {- \frac{x^{k}_{1} + \cdots + x_{q}^{k}}{k!}}
        \right)\;{\rm d} \cev{\bf x}_{q}
        = 
        (k!)^{{q}/{k}} \Gamma \left(1+\frac{1}{k}\right)
        ^{q}. 
    \end{equation*}
\end{lemma}

\begin{proof}
By making the change of variables \(w_{i} = x^{k}_{i}/k!\), for $1 \leq i \leq q$,  we notice that the integral at the right-hand side of (\ref{eq33}) is equal to 
    \begin{equation*}\label{eq:gam:order}
q!  (k!)^{{q}/{k}} \Gamma \left(1+\frac{1}{k}\right)
        ^{q}\int_{0}^{\infty} \int_{0}^{w_{1}} \cdots \int_{0}^{w_{q-1}} 
        \prod_{i=1}^{q} 
        \frac{e^{-w_{i}} w_{i}^{\frac{1}{k}-1}}{\Gamma(1/k)}
        {\rm d} \cev{\bf w}_{q}
        = (k!)^{{q}/{k}} \Gamma \left(1+\frac{1}{k}\right)
        ^{q}. 
    \end{equation*}
To see the last identity, we notice that the integral at the left-hand side is simply the
probability that \(G_{1} \ge G_{2} \ge \dots \ge G_{q}\), where
    \(G_{1},\dots,G_{q}\) are independent \(\text{Gamma}(1/k, 1)\) random variables, which is equal to \(1/q!\) since each order of \(G_{1},\dots,G_{q}\) is equally likely.
\end{proof}

\subsubsection{Split trees} 

The class of random split trees was first introduced by Devroye
\cite{Luc1999} to encompass many families of trees that are frequently used in algorithm analysis,
e.g., binary search trees and tries. Its exact construction is somewhat lengthy and we refer readers
to either the original algorithmic definition in \cite{Luc1999, holmgren2012} or the more probabilistic version in
\cite[Section 2]{Bro2012}. Informally speaking, a split tree $\dsT_{n}^{\rm sp}$ is constructed by
first distributing $n \in \dsN$ balls among the vertices of an infinite $b$-ary tree ($b \in \dsN
\setminus \{1\}$) and then removing all subtrees without balls. Each vertex in the infinite $b$-ary
tree is given a random non-negative split vector $\scV = (V_{1}, \dots, V_{b})$ such that
$\sum_{i=1}^{b}V_{i} = 1$ and $V_{i} \geq 0$, drawn independently from the same distribution. These
vectors affect how balls are distributed. In the study of split-trees, the following condition of
$\scV$ is often assumed (see, e.g., Holmgren \cite{holmgren2012}): \\

\noindent {\bf Condition A.} The split vector $\scV$ is permutation invariant. Moreover, $\dsP(V_{1} = 1) =\dsP(V_{1}=0) = 0$, and that $- \log(V_{1})$ is non-lattice. \\

\noindent Set $\mu \coloneqq b \dsE[-V_{1} \ln V_{1}] \in (0, \ln b)$. Devroye \cite{Luc1999} showed
that $\max_{ v \in \dsT_{n}^{\rm sp}} d_{n}(v) = O_{\rm p}(\ln n)$, that is, condition (a) in \Cref{Theo4}
with $a_{n} = \mu (\ln n)^{-1}$. Berzunza et al.\ \cite[Lemma 5 and Corollary 1]{Ber2019} have
shown that $\mu (\ln n)^{-1}(L_{n}(u_{1}), L_{n}(u_{1}, u_{2})) \inlaw (1, 2)$, as $n \rightarrow
\infty$. By a simple application of \cite[Corollary 1]{bertoin2013}, this implies that condition (b) in \Cref{Theo4} is
satisfied with $\zeta_{1} \equiv 1$. Notice that each vertex in $\dsT_{n}^{\rm sp}$ has at most $b$
children. Then it should be clear that condition (c) of \Cref{Theo4} follows from
\Cref{lemma6} since $\dsE[W_{i}(\dsT_{n}^{\rm sp})] \leq b^{i}$ for $i \in \dsZ_{\ge 0}$.
Therefore, \Cref{Theo4} implies that $\mu^{-1/k} n^{-1} (\ln n)^{1/k} \scK_{1}(\dsT_{n}^{\rm sp})
\inlaw Z_{1}$, as $n \rightarrow \infty$, where $Z_{1}$ is the random
variable whose law is determined entirely by its moments given in \eqref{eq33}. Furthermore, \Cref{lemma10} and \Cref{lemma11} imply that $\dsE[\scK_{r}(\dsT_{n}^{\rm sp}) ] = O(n (\ln n)^{-r/k})$ for
$r \in \{1, \dots, k\}$. Therefore, by the identity \eqref{eq32} and the Markov's inequality,
\begin{equation*}\label{eq:sp}
    \mu^{-1/k} n^{-1} (\ln n)^{1/k} \scK_{1}(\dsT_{n}^{\rm sp}) \inlaw 
    Z_{1} \equiv
    (k!)^{{1}/{k}} \Gamma \left(1+1/k\right)
    .
\end{equation*}

\subsubsection{Uniform random recursive trees} 

A uniform random recursive tree $\dsT_{n}^{\rm rr}$
is a random tree of $n \in \dsN$ vertices constructed recursively as follows: let $\dsT_{1}^{\rm
    rr}$ be the tree of a single vertex labelled $1$, given $\dsT_{n-1}^{\rm rr}$, choose a vertex in
$\dsT_{n-1}^{\rm rr}$ uniformly at random and attach a vertex labelled $n$ to the selected vertex as its
child, which give $\dsT_{n}^{\rm rr}$. The uniform random recursive tree is one of the most
studied random tree models. They appear for instance as simple epidemic models, or in computer
science as data structures. We refer to \cite[Chapter 6]{Drmota2} for background. Theorem
6.32 in \cite{Drmota2} shows that $\max_{ v \in \dsT_{n}^{\rm rr}} d_{n}(v) = O_{\rm p}(\ln n)$, that is, condition
(a) in \Cref{Theo4} is satisfied with $a_{n} = (\ln n)^{-1}$. From the results of Dobrow
\cite{Dobrow} (see also \cite[Section 2.5.5]{Drmota2}), it is not difficult to see that $(\ln
n)^{-1}(L_{n}(u_{1}), L_{n}(u_{1}, u_{2})) \inlaw (1, 2)$, as $n \rightarrow \infty$.
By a simple application of \cite[Corollary 1]{bertoin2013}, this implies that condition (b) in \Cref{Theo4} is satisfied with $\zeta_{1} \equiv 1$. By
\cite[Equation (11)]{Fuchs}, 
\begin{eqnarray*}
\dsE[W_{i}(\dsT_{n}^{\rm rr})] = \frac{(\ln n)^{i}}{\Gamma(1+1/(\ln n)) i!} (1+O(1/(\ln n)))
\end{eqnarray*}

\noindent uniformly for $n \geq 3$ and $1 \leq i \leq K \ln n$, for all $K \geq 1 $. Then it should
be clear that condition (c) of \Cref{Theo4} follows from \Cref{lemma6}. Therefore,
\Cref{Theo4} implies that $n^{-1} (\ln n)^{1/k} \scK_{1}(\dsT_{n}^{\rm rr})  \inlaw
Z_{1}$, as $n \rightarrow \infty$, where $Z_{1}$ is the random variable whose law is entirely
determined by its moments given in \eqref{eq33}. Furthermore, \Cref{lemma10} and \Cref{lemma11} imply that $\dsE[\scK_{r}(\dsT_{n}^{\rm rr}) ] = O(n (\ln n)^{-r/k})$ for $r \in \{1, \dots, k\}$.
Therefore, by the identity \eqref{eq32} and the Markov's inequality, 
\begin{equation*}\label{eq:rr}
    n^{-1} (\ln n)^{1/k} \scK_{1}(\dsT_{n}^{\rm rr}) \inlaw 
    Z_{1} \equiv
    (k!)^{{1}/{k}} \Gamma \left(1+1/k\right).
\end{equation*}

\subsubsection{Scale-free random trees} 

Scale-free random trees form a family of random trees that grow following a preferential attachment
algorithm, and are commonly used to model complex real-world networks; see M\'ori \cite{Mori2002}. A
scale-free random tree $\dsT_{n}^{\rm sf}$ is a random tree of $n \in \dsN$ vertices constructed
recursively as follows: Fix a parameter $\alpha \in (-1, \infty)$, and start from the tree
$\dsT_{1}^{\rm sf}$ that consists in a single edge connecting the vertices labelled  $1$ and $2$.
Suppose that  $T_{n}^{\rm sf}$ has been constructed for some $n \geq 1$, and for every $i \in \{1,
\dots, n+1\}$, denote by ${\rm deg}_{n}(i)$ the degree of the vertex $i$ in $T_{n}^{\rm sf}$. Then
conditionally given $T_{n}^{\rm sf}$, $T_{n+1}^{\rm sf}$ is built by adding an edge between
the new vertex $n+2$ and a vertex $v_{n}$ in $T_{n}^{\rm sf}$ chosen at random according to the law
\begin{eqnarray*}
 \dsP ( v_{n} = i| T_{n}^{\rm sf} ) = \frac{{\rm deg}_{n}(i)+ \alpha}{2n+\alpha(n+1)}, \hspace*{5mm} i \in \{1, \dots, n+1\}.
\end{eqnarray*}

\noindent The standard preferential attachment tree (also known as plane-oriented recursive tree) was made popular by Barab\'asi and Albert \cite{barabasi1999} and it
corresponds to the choice of \(\alpha = 0\).
On the other hand, if one lets $\alpha \rightarrow \infty$, then the algorithm yields a uniform
random recursive tree.  Janson \cite{janson2019} showed that scale-free random trees can also be
viewed as split trees with the branching factor \(b = \infty\).

Pittel \cite{Pittel1994} showed that $\max_{ v \in \dsT_{n}^{\rm sf}} d_{n}(v) = O_{\rm p}(\ln n)$,
that is, condition (a) in \Cref{Theo4} is satisfied with $a_{n} = (\beta \ln n)^{-1}$, where
$\beta \coloneqq (1+\alpha)/(2+\alpha)$. From the results of Borovkov and Vatutin \cite{Borov2006}
(see the bibliography therein for further references), it is not difficult to see that $ (\beta \ln
n)^{-1}(L_{n}(u_{1}), L_{n}(u_{1}, u_{2})) \inlaw (1, 2)$, as $n \rightarrow \infty$. By a simple
application of  \cite[Corollary 1]{bertoin2013}, this implies that condition (b) in \Cref{Theo4}
is satisfied with $\zeta_{1} \equiv 1$.
Hwang \cite[Equation 8]{hwang2007} showed that, for \(\alpha=0\), i.e., for the standard
preferential attachment tree,
\begin{equation}\label{eq:sf:w}
    \dsE[W_{i}(\dsT^{\rm sf}_{n})]
    =
    \frac{\sqrt{\pi n} 2^{1-i}  (\ln n)^{i-1}}{\Gamma (i) \left({2 i}/{(\ln n)}+1\right)
        \Gamma \left({i}/{(\ln n)}+1\right)}
    \left( 1+ O\left( {1}/{(\ln n)} \right) \right)
    ,
\end{equation}
uniformly for \(1 \le i \le K \ln n\) for all \(K \ge 1\). Thus by an argument similar to that for uniform
random recursive trees, we have for \(\alpha = 0\),
\begin{equation}\label{eq:sf}
  2^{-1/k} n^{-1} (\ln n)^{1/k} \scK(\dsT_{n}^{\rm st}) \inlaw 
    Z_{1} \equiv
    (k!)^{{1}/{k}} \Gamma \left(1+1/k\right)
    .
\end{equation}

\noindent\textbf{Open problem.} To apply \Cref{Theo4} to general scale-free trees, we need an estimate of \(\E{W_{i}(\dsT^{\rm sf}_{n})}\) for all \(\alpha > -1\), which is currently missing in the
literature. Thus we leave it as an open problem that an estimation similar to  (\ref{eq:sf:w}) holds for all \(\alpha >-1\). This would imply that the convergence in  (\ref{eq:sf}) holds for all scale-free trees. 

\begin{remark}
 In all previous examples of Section \ref{sect43}, the limit distributions found here are all
degenerate. However, we conjecture that another normalization should yield to non-degenerate
limits. This is known to be the case, when $k=1$, for complete binary trees (Janson
\cite{Janson2004}), recursive trees (Drmota et al.\ \cite{Drmota2009}, Iksanov and Möhle
\cite{Iksanov2007}), binary search trees (Holmgren \cite{Holmgren2010}) and split trees (Holmgren
\cite{Holmgren2011}). In the general case $k \geq 1$, Cai and Holmgren \cite{Cai20182} obtained also
a weak limit theorem in the case of complete binary trees suggesting that our conjecture must be
true. 
\end{remark}

\subsubsection{Mixture of regular trees} 

Our next example provides a method to build trees that fulfill the conditions of Theorem \ref{Theo4}
where the random variables $\zeta_{1}, \zeta_{2}, \dots$ in the hypotheses are not constants.
Basically, the procedure consists of gluing trees which satisfy the assumptions of Theorem
\ref{Theo4}. In this example, we consider a mixture of complete regular trees but one may consider
other families of trees as well. For a fixed integer $m \geq 1$, let $(d_{i})_{i=1}^{m}$ denote a
positive sequence of integers. Next, for $i = 1, \dots, m$, let $h_{i}(n): \mathbb{R}_{+}
\rightarrow \mathbb{R}_{+}$ be a function with $\lim_{n \rightarrow \infty} h_{i}(n) = \infty$. Let
$T_{n_{i}}^{(d_{i})}$ be a complete $d_{i}$-regular tree with height $\lfloor h_{i}(n) \rfloor$.
Since there are $d_{i}^{j}$ vertices at distance $j=0, 1, \dots, \lfloor h_{i}(n) \rfloor$ from the
root, its size is given by
\begin{eqnarray*}
 n_{i} = n_{i}(n) = d_{i} (d_{i}^{\lfloor h_{i}(n) \rfloor}-1)/(d_{i}-1).
\end{eqnarray*}

\noindent In particular, one can check that each tree $T_{n_{i}}^{(d_{i})}$ fulfills the assumptions
in Theorem \ref{Theo4} with $a_{n} = \ln n_{i}$ and $\zeta_{1} = (\ln d_{i})^{-1}$; note that
condition (c) in Theorem \ref{Theo4} follows from \Cref{lemma6} and the fact that the number of
descendants of each vertex is bounded. Now imagine that we merge all the $m$ regular trees into
one common root. This leads us to a new tree $T_{n}^{(d)}$ of size $n = \sum_{i=1}^{m} n_{i} + 1
-m$. Assume further that $n_{1} \sim n_{2} \sim \cdots \sim n_{m}$, as $n \rightarrow \infty$. Then,
we observe that the probability that a vertex of $T_{n}^{(d)}$ chosen uniformly at random belongs to
the tree $T_{n_{i}}^{(d_{i})}$ converges when $n \rightarrow \infty$ to $1/m$. Then, one readily
checks that this new tree satisfies the hypotheses in Theorem \ref{Theo4} with $a_{n} = \ln n$ and
$\zeta_{1}, \zeta_{2}, \dots$ are i.i.d.\ random variables uniformly distributed in the set
$\{1/ \ln d_{1}, \dots, 1 / \ln d_{m} \}$. To see this, note that the probability that a uniform chosen vertex of $T_{n}^{(d)}$ belongs to $T_{n_{i}}^{(d_{i})}$ converges to $1/m$.

\paragraph{Acknowledgements.}
This work is supported by the Knut and Alice Wallenberg
Foundation, a grant from the Swedish Research Council and The Swedish Foundations' starting grant from Ragnar S\"oderbergs Foundation.


\providecommand{\bysame}{\leavevmode\hbox to3em{\hrulefill}\thinspace}
\providecommand{\MR}{\relax\ifhmode\unskip\space\fi MR }
\providecommand{\MRhref}[2]{%
  \href{http://www.ams.org/mathscinet-getitem?mr=#1}{#2}
}
\providecommand{\href}[2]{#2}

\end{document}